\documentclass[final,3p,11pt]{elsarticle}

\usepackage{verbatim}
\usepackage{mathrsfs}
\usepackage{amsfonts}
\usepackage{amsmath}
\usepackage{amssymb}
\usepackage{bbm,dsfont}
\usepackage{mathrsfs}
\usepackage{graphicx}
\usepackage[usenames,dvipsnames]{xcolor}
\usepackage{enumerate}
\usepackage{theorem}

\allowdisplaybreaks

\newcommand{\ea}{\end{array}}
\newtheorem{thm}{Theorem}[section]
\newtheorem{prop}{Proposition}[section]

\newtheorem{lem}{Lemma}[section]
\newtheorem{defn}{Definition}[section]

\begin{document}

\begin{frontmatter}

\title{Well-Posedness of Stochastic Chemotaxis System}

\author[mymainaddress]{Yunfeng Chen}
\ead{cyf01@mail.ustc.edu.cn}

\author[mymainaddress]{Jianliang Zhai}
\ead{zhaijl@ustc.edu.cn}

\author[mymainaddress,mysecondaryaddress]{Tusheng Zhang}
\ead{Tusheng.Zhang@manchester.ac.uk}

\address[mymainaddress]{School of Mathematical Sciences, University of Science and Technology of China, Hefei, 230026, China}
\address[mysecondaryaddress]{ Department of Mathematics, University of Manchester,
 Oxford Road, Manchester, M13 9PL, UK.}

\begin{abstract}
In this paper, we establish the existence and uniqueness of
solutions of elliptic-parabolic stochastic Keller-Segel systems. The solution is
obtained through a carefully designed localization procedure together with some a priori estimates. Both noise of linear growth and nonlinear noise
are considered. The $L^p$ It\^{o} formula plays an important role.
\end{abstract}

\end{frontmatter}


\section{Introduction}
\setcounter{equation}{0}
The purpose of this paper is to establish the well posedness of  the following stochastic partial differential equation
(SPDE) on $\mathcal{O} \times (0, T)$:
\begin{equation}\label{m}
\left\{
    \begin{array}{lr}
du = \Delta u \, dt - \nabla \cdot (\chi u \nabla v)  \, dt + g(u) \, dt + \sum_{i=1}^{\infty}\sigma_i(u) \,  \, dW^i_t, \quad  \\
0 = \Delta v + u - v, \quad
\end{array}
\right.
\end{equation}
with the initial-boundary condition:
\begin{equation*}
	\left\{
	\begin{array}{lr}
		\frac{{\partial }}{{\partial \boldsymbol{n}}}u  = 0, \quad \frac{{\partial }}{{\partial \boldsymbol{n}}}v = 0, &\text{for } x \in\partial \mathcal{O},\\
		u(x,0) = u_0(x), \quad &\text{for } x \in \mathcal{O},\
	\end{array}
	\right.
\end{equation*}
where $\mathcal{O} \subset \mathbb{R}^2$ is a bounded domain with a smooth boundary, $\{W^k_t\}_{k\in\mathbb{N}}$ represents a family of independent real-valued Brownian motions on a given complete, filtered probability space $(\Omega,\mathcal{F},\{\mathcal{F}_t\}_{t\geq0},\mathbb{P})$, $\chi > 0$ denotes the chemotactic sensitivity, and the symbol $\frac{{\partial}}{{\partial \boldsymbol{n}}}$ denotes the derivative with respect to the outer normal of $\partial \mathcal{O}$.

The initial function $u_0$ is assumed to be nonnegative and belong to $C(\bar{\mathcal{O}})$ with mass $m_0=\int_\mathcal{O} u_0(x) \, dx$.
The function $g \in C^{1}([0,\infty)))$ satisfies appropriate conditions specified later.
The logistic function $g(u) = \mu u(1 - u)$ can be considered as a prototype for $g$. It is commonly used in various mathematical modeling, including logistic regression and population growth.

The problem (\ref{m}) is a random perturbation of the  Keller–Segel model of chemotaxis. The function $u(x,t):\mathcal{O}\times\mathbb{R}^+\rightarrow\mathbb{R}^+$ describes the particle density at time $t$ and position
$x\in \mathcal{O}$, while $v(x,t):\mathcal{O}\times\mathbb{R}^+\rightarrow\mathbb{R}^+$ represents the density of the external chemical substance. Chemotaxis is the directed movement of cells in response to chemical gradients in their environment. When cells move towards an increasing signal concentration, it is known as chemoattraction, while moving away from the increasing signal concentration is called chemorepulsion. The concept of chemotaxis was introduced by Keller and Segel \cite{KS}. For detailed examples highlighting the biological relevance of chemotaxis, see the paper \cite{OTYM}.
In the past two decades, many people have thoroughly analyzed various mathematical models of partial differential equations related to chemotaxis.  Investigations have focused on whether the solutions to these equations exhibit blow-up phenomena within a finite time or exist globally, as discussed in \cite{HV,HP,H1,H2,H3,M1,M2} and reference therein.

 Taking into account the effect of the random environment and random external forces,  stochastic Keller–Segel models were investigated recently.  
In \cite{FGL}, the authors established the local existence of solutions for a class  of nonlinear SPDEs, including the stochastic Keller–Segel equations on a torus. In \cite{HQ}, the authors showed that for initial conditions with sufficiently small $L^4$ norms, the stochastic Keller–Segel model driven by divergence-type noise admits a global weak solution.
In \cite{MST}, it is shown that if  initial mass is sufficiently large, the solutions of the stochastic Keller–Segel system blow up in finite time  with probability one. On the other hand, with the appearance of the logistic term $g$ it is known(see
e.g.\cite{IM}) that the global solutions of the deterministic Keller–Segel systems  exist for all initial masses.
In recent works (\cite{FL, HLL, RTW, RZZ, TY}), it has been demonstrated that the appropriate introduction of nonlinear noise can prevent the explosion of certain stochastic equations. Notably, the authors in \cite{RSG} established the global existence of solutions for the Keller-Segel model in Gevrey-type Fourier-Lebesgue spaces with significantly high probability. Inspired by these findings, we also investigate a chemotaxis equation driven by a class of nonlinear noise in the final section. We will show that this nonlinear noise can effectively play the role of the logistic source $g$ in preventing the solution from exploding.

\vskip 0.3cm

The purpose of this paper is to obtain the well-posedness of the stochastic Keller–Segel system with a logistic type source or a nonlinear noise. The appearance of the random noise makes the solution of the system singular. The methods of solving the deterministic system do not apply. We first prove the existence of a local maximum solution by a designed localizing procedure and the Banach fixed point theorem.
To obtain the global existence, we prove a uniform a priori  estimate for the local solutions. To this end, we prove a $L^p$ It\^{o} formula for the local solution using approximation arguments.

\vskip 0.3cm

The organization of the paper is as follows. In Section \ref{section2}, we introduce the notations, assumptions and the main results of the paper.  Section 3 is devoted to the existence and uniqueness of local solutions. In Section 4, the global solutions of stochastic Keller–Segel system with linear growth noise are obtained.  In Section 5, we prove the global existence and uniqueness of the solution to stochastic chemotaxis equation driven by nonlinear noise.

\section{Statement of the Main Result}\label{section2}
\setcounter{equation}{0}
We begin by introducing the notations used throughout the paper. The external random driving force, denoted by $\{W^k_t\}_{k\in\mathbb{N}}$, represents a family of independent real-valued Brownian motions on a given complete, filtered probability space $(\Omega,\mathcal{F},\{\mathcal{F}_t\}_{t\geq0},\mathbb{P})$.

For $k, q, r \in [1,\infty]$, we denote the norm in the $L^q(\mathcal{O})$ space by $\|\cdot\|_q$. $\|\cdot\|_{r,q,T}$ denotes the norm in the space $L^{r,q,T}$ defined as $L^r([0,T];L^q(\mathcal{O}))$. Moreover, $W^{k,q}(\mathcal{O})$ denotes the Sobolev space of functions whose distributional derivatives up to order $k$ belong to $L^q(\mathcal{O})$.

The space $\ell^2$ is defined as the set of all sequences of real numbers $(x_n)$ satisfying $\sum x_n^2<\infty$. Additionally, $C(\overline{\mathcal{O}})$ will refer to the space of continuous functions on the closure of $\mathcal{O}$.

For $p\in(1,\infty)$, let $A := A_p$ denote the realization of $-\Delta$ in $L^p(\mathcal{O})$ with domain $\mathcal{D}(A) = \{\varphi \in W^{2,p}(\mathcal{O}) : \frac{\partial\varphi}{\partial \boldsymbol{n}}|_{\partial\mathcal{O}} = 0\}$. The operator $A+1$ possesses fractional powers $(A+1)^\beta, \beta\geq 0$, the domains of which have the embedding property:
  \begin{align*}
      D((A_p+1)^\beta)\hookrightarrow C(\Bar{\mathcal{O}})\ \ \text{if} \ \beta>\frac{1}{p}.
  \end{align*}
  The fact that the spectrum of $A$ is a $p$-independent countable set of positive real numbers entails the following consequences. For the proofs of this lemma, we refer to \cite{HW,W2010}.
  \begin{lem}
      Let $\left(e^{-tA}\right)_{t\geq 0}$ be the Neumann heat semigroup in $\mathcal{O}$.
       Let $\nu_1$ denote the first nonzero eigenvalue of $A$ on $\mathcal{O}$ under Neumann boundary condition. Then, there exists a constant $C$, depending only on $\mathcal{O}$, such that:
      \begin{itemize}
  \item[(\romannumeral 1)] If $1\leq p\leq \infty$, for $t>0$
  \begin{align}
      \|(A+1)^\beta e^{-tA} \omega\|_{p}\leq Ct^{-\beta}e^{-\nu_1t}\|\omega\|_{p},\label{A1}\\
       \|\nabla e^{-tA} \omega\|_{p}\leq C(1+t^{-\frac{1}{2}})e^{-\nu_1t}\|\omega\|_{p},\label{A2}
  \end{align}
holds for all $\omega\in L^{p}(\mathcal{O})$.
  \item[(\romannumeral 2)] If $p\in [2,\infty)$, for $t>0$
  \begin{align}\label{A3}
       \|\nabla e^{-tA} \omega\|_{p}\leq Ce^{-\nu_1t}\|\omega\|_{W^{1,p}(\mathcal{O})},
  \end{align}
is valid for all $\omega\in W^{1,p}(\mathcal{O})$.
		
\item[(\romannumeral 3)]
If $\beta>0$, $p\in (1,\infty), t>0$, then the operators $(A+1)^\beta e^{-tA}\nabla\cdot$ and $e^{-tA}\nabla\cdot$ both possess extension to  operators from $L^{p}(\mathcal{O})$ into $L^{p}(\mathcal{O})$ which fulfill
  \begin{align}
      &\|(A+1)^\beta e^{-tA}\nabla\cdot w\|_{p}\leq C(\epsilon)t^{-\frac{1}{2}-\beta-\epsilon}e^{-\nu_1 t}\|w\|_{p}, \label{A4}\\
      &\| e^{-tA}\nabla\cdot w\|_{p}\leq Ct^{-\frac{1}{2}}e^{-\nu_1 t}\|w\|_{p},\label{A5}
  \end{align}
  for all $\omega\in (L^p(\mathcal{O}))^2$, any $\epsilon>0$.
	\end{itemize}
  \end{lem}
\vskip 0.3cm

Given $f\in L^p(\mathcal{O})$, let $w$ be a weak solution to the boundary value problem:
\begin{equation}
\left\{
\begin{array}{lr}
	-\Delta w+w=f\ \ \ &\text{in}\ \mathcal{O},\notag\\
	\frac{\partial w}{\partial \boldsymbol{n}}=0\ \ \ &\text{on}\ \partial \mathcal{O}.
\end{array}
\right.
\end{equation}
Then the solution $w$ can be expressed as
\begin{align*}
    w(x)=G*f(x)=\int_\mathcal{O} G(x,y)f(y)dy  \ \ \ \text{for}\ x\in \mathcal{O},
\end{align*}
where $G(x,y)$ is the Green's function of $-\Delta+1$ on $\mathcal{O}$ with homogeneous Neumann boundary conditions. Using the $L^r$-estimate(for $2<r<\infty)$ for elliptic equations(see \cite{T}), we have
\begin{align*}
    \|w\|_{2,r}\leq \|f\|_r,
\end{align*}
and the Sobolev embedding $W^{1,r}(\mathcal{O})\hookrightarrow C(\bar{\mathcal{O}})$(see\cite{H}) implies also that
\begin{align}\label{w}
    \|w\|_\infty+\|\nabla w\|_\infty\leq C\|f\|_r.
\end{align}
\vskip 0.4cm
Consider the stochastic chemotaxis system on $\mathcal{O} \times (0, T)$:
\begin{equation}\label{main}
\left\{
    \begin{array}{lr}
du = \Delta u \, dt - \nabla \cdot (\chi u \nabla v)  \, dt + g(u) \, dt + \sum_{i=1}^{\infty}\sigma_i(u) \,  \, dW^i_t, \quad  \\
0 = \Delta v + u - v, \quad
\end{array}
\right.
\end{equation}
with the initial-boundary condition:
\begin{equation*}
	\left\{
	\begin{array}{lr}
		\frac{{\partial }}{{\partial \boldsymbol{n}}}u  = 0, \quad \frac{{\partial }}{{\partial \boldsymbol{n}}}v = 0, &\text{for } x \in\partial \mathcal{O},\\
		u(x,0) = u_0(x), \quad &\text{for } x \in \mathcal{O}.\
	\end{array}
	\right.
\end{equation*}
For the $\ell^2$-valued  mapping $\sigma=(\sigma_k)_{k=1}^\infty$, where $\sigma_k:\mathbb{R}\rightarrow \mathbb{R}$, we introduce the following assumption:
\begin{itemize}
	\item[{\bf (H-1)}.] There exists a positive constants $K$  such that for all $z_1,z_2,z\in \mathbb{R}$,
	\begin{itemize}
    \item[\textbullet] $\sigma(0) = 0$,
    \item[\textbullet] $\| \sigma(z)\|_{\ell^2}\leq K(|z|+ 1)$,
    \item[\textbullet] $\| \sigma(z_1)-\sigma(z_2)\|_{\ell^2}\leq K |z_1 - z_2|$.
\end{itemize}	
\end{itemize}
\vskip 0.3cm
Regarding the nonlinear function $g$, we introduce the condition:
\begin{itemize}
	\item[{\bf (H-2)}.] \begin{align}\label{g}
g(0)\geq 0,\ \ \text{and}\ \  g(s) \leq c_1 - \mu s^2, \quad \forall s \geq 0,
\end{align}
where $c_1$ and $\mu$ are positive constants.	
\end{itemize}
\vskip 0.3cm
\begin{defn}\label{mild solution}
	We say that $(u,v)$ is a mild solution of system (\ref{main}) if $(u,v)$ is a progressively measurable stochastic process with values in $C(\bar{\mathcal{O}})\times W^{1,\infty}({\mathcal{O}})$ that satisfies:
	\begin{align*}
		&u(t) = e^{-tA}u_0 - \int_0^t e^{-(t-s)A}\left( \nabla \cdot \left( \chi u(s) \nabla v(s) \right) - g(u(s)) \right)\, ds + \sum_{i=1}^\infty\int_0^t e^{-(t-s)A} \sigma_i(u(s)) \,  \, dW_s^i, \\
        &v(t)=G*u(t),
	\end{align*}
for all $t\geq 0$.
\end{defn}

Here is the first main result, whose proof is given in Section 4.
\begin{thm}\label{Thm main}
	Suppose the assumptions {\bf (H-1)}, {\bf (H-2)} are in place with some $\mu>\frac{\chi}{2}$. Then
for any nonnegative $u_0\in C(\bar{\mathcal{O})}$, the system (\ref{main}) possesses a unique
nonnegative  mild solution $(u,v)$.
\end{thm}

Next we consider stochastic chemotaxis system driven by a class of nonlinear noise:

\begin{equation}\label{main2}
\left\{
    \begin{array}{lr}
du = \Delta u \, dt - \nabla \cdot (\chi u \nabla v)  \, dt + g(u) \, dt +  \sum_{i=1}^{\infty}b_i\|u\|_q^ru \,  \, dW^i_t, \quad  \\
0 = \Delta v + u - v, \quad
\end{array}
\right.
\end{equation}
with the same initial-boundary condition:
\begin{equation*}
	\left\{
	\begin{array}{lr}
		\frac{{\partial }}{{\partial \boldsymbol{n}}}u  = 0, \quad \frac{{\partial }}{{\partial \boldsymbol{n}}}v = 0, &\text{for } x \in\partial \mathcal{O},\\
		u(x,0) = u_0(x), \quad &\text{for } x \in \mathcal{O}.\
	\end{array}
	\right.
\end{equation*}

Regarding the coefficients,  we introduce following conditions:

\vskip 0.3cm
\noindent{\bf (A-1)}.
	\begin{itemize}
    \item[\textbullet] $\sum_{i=1}^\infty b_i^2< \infty$
    \item[\textbullet] $r> \frac{2\vee n-1}{2}$, $q\geq 2r$ and $q> \frac{2(2\vee n-1)r}{2r-2\vee n+1}$
\end{itemize}
\vskip 0.3cm
\noindent {\bf (A-2)}.
\begin{align}\label{g'}
    g(0)\geq 0,\ \ \text{and}\ \  |g(s)| \leq c_2 + \mu' |s|^n, \quad \forall s \geq 0,
\end{align}
where $c_2$, $\mu'$ and $n$ are positive constants.

%
\vskip 0.4cm
The mild solution  $(u,v)$ of the system (\ref{main2}) can be defined analogously as (\ref{main}).
\vskip 0.3cm
 The following is the second main result.
 \begin{thm}\label{Thm main2}
	Suppose the assumptions {\bf (A-1)}, {\bf (A-2)} are in place. Then for any nonnegative $u_0\in C(\bar{\mathcal{O})}$, the system (\ref{main2}) possesses a unique
nonnegative  mild solution $(u,v)$.
\end{thm}

The proof of the above result is given in Section 5.

\section{Existence and Uniqueness of Local Solutions}\label{section3}
\setcounter{equation}{0}
In this section, we will show that there exists a unique local solution to the stochastic chemotaxis system.
We first state a lemma which  provides estimates for the stochastic convolution $M_t: =\int_0^t e^{-(t-s)A} h(s) \,  \, dW_s$, which will be frequently used. A similar result was proved  in \cite{DG} where no boundary conditions are imposed. However, the proofs there can be adapted to the current situation with Neumann boundary condition. We omit the proof here.
\begin{lem}\label{M}
	Suppose $h=(h_k)_{k=1}^\infty$, is an $\ell^2$-valued
progressively measurable function on $\Omega\times[0,T]\times\mathcal{O}$ such that $$\mathbb{E}\|\|h\|_{\ell^2}\|_{2,2,T}^2< \infty.$$
 Then $M(t)=\int_0^te^{-(t-s)A}h(s) \, dW_s$ is the solution to the following SPDE:
	\begin{equation*}
		\begin{cases}
dM(t) = -A M(t)  \, dt + h(t) \,  \, dW_t, \\
M(0) = 0,
\end{cases}
	\end{equation*}
	in $L^2(\mathcal{O})$, and has the $L^\infty$-regularity:
	\begin{align}\label{Mt}
		\mathbb{E} \sup_{t\in[0,T\wedge\tau]}\|M(t)\|_\infty^\eta\leq C(1+T)^{C'}\mathbb{E}\|\|h\|_{\ell^2}\|_{2r,2q,T\wedge\tau}^\eta.
	\end{align}
for any $\eta>0$, $r,q\in(1,\infty]$ with $\frac{1}{r}+\frac{1}{q}<1$ and stopping time $\tau$. The constants $C,C'$ may depend on $\eta, r, q $.
\end{lem}

\subsection{Existence of Local Solutions}\label{subsection4.1}
Observe that once we have $u(t)\in C(\bar{\mathcal{O}})$, we can conclude that $v\in W^{1,\infty}$ using the regularity of the elliptic equation (\ref{w}).
Introduce the following space
$$
\Upsilon^u_t:=L^\infty([0,t],C(\bar{\mathcal{O}}))
$$
with the corresponding norm given by
$$
\|u\|_{\Upsilon^u_t}=\sup_{s\in[0,t]}\|u(s)\|_{\infty}.
$$

\begin{defn}\label{def local solution}
	We say that $(u,v,\tau)$ is a local mild solution to the system (\ref{main}) if
	\begin{enumerate}
    \item$\tau$ is a stopping time, and $(u, v)$ is a progressively measurable stochastic process that takes values in $C(\bar{\mathcal{O}})\times W^{1,\infty}(\mathcal{O})$.

    \item There exists a non-decreasing sequence of stopping times $\{\tau_l, l \geq 1\}$ such that $\tau_l \uparrow \tau$ almost surely as $l \to \infty$, and for each $l$, the pair $(u(t), v(t))$ satisfies:
   \begin{align*}
    u(t) &= e^{-tA} u_0 - \int_0^{t} e^{-(t - s) A} \left( \nabla \cdot \left( \chi u(s) \nabla v(s) - g(u(s)\right)\right)\, ds
    + \int_0^t  e^{-(t - s) A} \sigma(u(s)) \, dW_s,\\
    v(t) &= G * u(t),
    \end{align*}
    for all $t \leq \tau_l$.
\end{enumerate}
\end{defn}


\begin{thm}\label{thm l}
Under the assumption ({\bf H-1}), for any $g \in C^{1}([0,\infty))$ and nonnegative $u_0\in C(\bar{\mathcal{O})}$,	there exists a nonnegative local mild solution to the system (\ref{main}).
\end{thm}
\begin{proof}
The existence proof utilizes a cutoff procedure along with a contraction mapping argument. We modify the coefficients in system (\ref{main}), and consider the following SPDE:
\begin{equation}\label{cut}
du = \Delta u  \, dt - \theta_m(\|u\|_{\Upsilon^u_t})\nabla \cdot (\chi u \nabla (G * u))  \, dt + \theta_m(\|u\|_{\Upsilon^u_t})g(u) \, dt + \sigma(u) \,  \, dW_t, \quad
\end{equation}
Here $\theta_m(\cdot)=\theta(\frac{\cdot}{m})$ with $\theta\in C^2([0,\infty),[0,1])$ being a function such that:
\begin{itemize}
\item[(1)] $\theta(r)=1$,\ \ for\ $r \in [0, 1]$,
\item[(2)] $\theta(r)=0$,\ \ for\ $r > 2$,
\item[(3)] $\sup_{r\in[0,\infty)}|\theta'(r)|\leq C<\infty$.
	\end{itemize}

Let $S_T$ be the space of all $\{\mathcal{F}_t\}_{t\in[0,T]}$-adapted, $C(\bar{\mathcal{O}})$-valued  stochastic processes $u(t)$ such that
$$\|u\|_{S_T}:=\mathbb{E}\Big(\|u\|_{\Upsilon^u_T}\Big)<\infty.$$
Then $S_T$ equipped with the norm $\|\cdot\|_{S_T}$ is a Banach space.

	We claim that the map
 \begin{align*}
     \Phi(u)(t):=e^{-tA}u_0&-\int_0^te^{-(t-s)A}\theta_m(\|u\|_{\Upsilon^u_s})\big(\nabla\cdot(u\nabla (G*u))-g(u)\big)(s) \, ds\\
     &+\int_0^te^{-(t-s)A}\sigma(u(s)) \, dW_s,
 \end{align*}
 for $t\in [0, T ]$, is a contraction mapping from $S_T$ into itself if $T$ is sufficiently small.

To demonstrate this, we first note that for $u\in S_T$, the function $\Phi(u)$  takes values in $C(\bar{\mathcal{O}})$ because $u_0\in C(\bar{\mathcal{O}})$, $\sigma_k\in C(\mathbb{R})$, and $e^{-tA}$ is strongly continuous in $C(\bar{\mathcal{O}})$.

Consider some $p>2$, and fix $\beta\in (\frac{1}{p},\frac{1}{2})$. Since $D((A_p+1)^\beta)\hookrightarrow C(\bar{\mathcal{O}})$, we have
\begin{align*}
    	\|\Phi(u)(t)\|_\infty
    	\leq& \|e^{-tA}u_0\|_\infty+ C \int_0^t\|(A_p+1)^\beta e^{-(t-s)A}\theta_m(\|u\|_{\Upsilon^u_s})\nabla\cdot(u\nabla (G*u))(s)\|_{p} \, ds\\
    	&+\left\|\int_0^te^{-(t-s)A}\theta_m(\|u\|_{\Upsilon^u_s})g(u(s)) \, ds\right\|_\infty+\left\|\int_0^t e^{-(t-s)A}\sigma(u(s)) \, dW_s\right\|_{\infty}.
    \end{align*}
Namely,
\begin{align}\label{Phi}
\|\Phi(u)(t)\|_\infty \leq C \|u_0\|_\infty + I_1(t) + I_2(t) + I_3(t),
\end{align}
where
\begin{align*}
&I_1(t) = C \int_0^t\|(A_p+1)^\beta e^{-(t-s)A}\theta_m(\|u\|_{\Upsilon^u_s})\nabla\cdot(u\nabla (G*u))(s)\|_{p} \, ds,\\
&I_2(t) = \left\|\int_0^te^{-(t-s)A}\theta_m(\|u\|_{\Upsilon^u_s})g(u(s)) \, ds\right\|_\infty,\\
&I_3(t) = \left\|\int_0^t e^{-(t-s)A}\sigma(u(s)) \, dW_s\right\|_{\infty}.
\end{align*}

To estimate $I_1(t)$, we use the estimate (\ref{A4}) with $\epsilon \in (0, \frac{1}{2} - \beta)$ and (\ref{w}) to obtain
    \begin{align*}
        I_1(t)\leq C_{\epsilon}\int_0^t(t-s)^{-\beta-\frac{1}{2}-\epsilon}\theta_m(\|u\|_{\Upsilon^u_s})\|u(s)\|_\infty^2 \, ds\leq C_{\epsilon}\int_0^t(t-s)^{-\beta-\frac{1}{2}-\epsilon}m^2 \, ds\leq C_{\epsilon}m^2t^{\frac{1}{2}-\beta-\epsilon}.
    \end{align*}

Due to the parabolic maximum principle, we can bound $I_2(t)$ as follows
    \begin{align*}
        I_2(t)\leq \int_0^te^{-(t-s)A}g_m \, ds\leq g_mt,
    \end{align*}
 where $g_m=\sup_{0\leq s\leq m}|g(s)|$.

For the term $I_3(t)$, applying Lemma \ref{M} and choosing $\eta = 1$, $r = \infty$, and $q = \infty$ as in (\ref{Mt}), we get
\begin{align*}
    \mathbb{E}\sup_{t\in[0,T]} I_3(t)\leq C\mathbb{E}\sup_{t\in[0,T]}\|\|\sigma(u(t))\|_{\ell^2}\|_\infty\leq C\left(\mathbb{E}\sup_{t\in[0,T]}\|u(t)\|_\infty+1\right).
\end{align*}

Combining these estimates with (\ref{Phi}), we obtain
\begin{align}\label{PHI}
    \mathbb{E}\sup_{t\in[0,T]} \|\Phi(u)(t)\|_\infty\leq C_T\left(\|u_0\|_\infty+m^2+g_m+\mathbb{E}\sup_{t\in[0,T]} \|u(t)\|_\infty\right).
\end{align}

    This shows that $\Phi$ maps $S_T$ into itself.

Next, we will show that if $T > 0$ is sufficiently small, then $\Phi$ can be made a contraction on $S_T$.

Let $u_1,u_2\in S_T$. We have:
\begin{align}\label{Phi1-2}
    \|\Phi(u_1)(t)-\Phi(u_2)(t)\|_\infty
    \leq J_1(t)+J_2(t)+J_3(t),
\end{align}
where
\begin{align*}
&J_1(t) = \left\|\int_0^t e^{-(t-s)A}\nabla\cdot\left(\theta_m(\|u_1\|_{\Upsilon^u_s})u_1\nabla (G*u_1)-\theta_m(\|u_2\|_{\Upsilon^u_s})u_2\nabla (G*u_2)\right)(s) \, ds\right\|_{\infty},\\
&J_2(t) = \left\|\int_0^te^{-(t-s)A}\big(\theta_m(\|u_1\|_{\Upsilon^u_s})g(u_1(s))-\theta_m(\|u_2\|_{\Upsilon^u_s})g(u_2(s))\big) \, ds\right\|_\infty,\\
&J_3(t) = \left\|\int_0^t e^{-(t-s)A}\big(\sigma(u_1(s))-\sigma(u_2(s))\big) \, dW_s\right\|_{\infty}.
\end{align*}

Similar to the estimate of $I_1$, we have
\begin{align}\label{J1 1}
    J_1(t)&\leq C \int_0^t\left\|(A_p+1)^\beta e^{-(t-s)A}\nabla\cdot\left(\theta_m(\|u_1\|_{\Upsilon^u_s})u_1\nabla (G*u_1)-\theta_m(\|u_2\|_{\Upsilon^u_s})u_2\nabla (G*u_2)\right)
    (s)\right\|_{p} \, ds\nonumber\\
   &\leq C_{\epsilon}\int_0^t(t-s)^{-\beta-\frac{1}{2}-\epsilon}\left\|(\theta_m(\|u_1\|_{\Upsilon^u_s})u_1\nabla (G*u_1)-\theta_m(\|u_2\|_{\Upsilon^u_s})u_2\nabla (G*u_2))
    (s)\right\|_p \, ds.
\end{align}

Define:
$$H(s)=\left\|\left(\theta_m(\|u_1\|_{\Upsilon^u_s})u_1\nabla (G*u_1)-\theta_m(\|u_2\|_{\Upsilon^u_s})u_2\nabla (G*u_2)\right)(s)\right\|_p.$$

    If $\|u_1\|_{\Upsilon^u_s}\vee\|u_2\|_{\Upsilon^u_s}\leq 2m$, then
    \begin{align*}
        H(s)\leq& \left\|\left(\theta_m(\|u_1\|_{\Upsilon^u_s})u_1\nabla (G*u_1)-\theta_m(\|u_2\|_{\Upsilon^u_s})u_1\nabla (G*u_1)\right)(s)\right\|_p \nonumber\\
    &+\left\|\left(\theta_m(\|u_2\|_{\Upsilon^u_s})u_1\nabla (G*u_1)-\theta_m(\|u_2\|_{\Upsilon^u_s})u_2\nabla (G*u_2)\right)(s)\right\|_p \nonumber\\
    \leq &
    \left|\theta_m(\|u_1\|_{\Upsilon^u_s})-\theta_m(\|u_2\|_{\Upsilon^u_s})\right|\left\|u_1\nabla (G*u_1)
    \right\|_p+\left\|u_1\nabla (G*u_1)-u_2\nabla (G*u_2)\right\|_p\nonumber\\
    \leq &
    \frac{C}{m}(\|u_1-u_2\|_{\Upsilon^u_s})\|u_1\|_\infty\|G*u_1\|_{1,p}+\|u_1\|_\infty\|G*u_1-G*u_2\|_{1,p}+\|u_1-u_2\|_\infty\| G*u_2\|_{1,p}\nonumber\\
    \leq&
    Cm(\|u_1-u_2\|_{\Upsilon^u_s}).
    \end{align*}
Here, we have used the fact that $\sup_{r\in[0,\infty)}|\theta_m^{\prime}(r)|\leq \frac{C}{m}$.

  If $\|u_1\|_{\Upsilon^u_s}\wedge\|u_2\|_{\Upsilon^u_s}> 2m$, then:
  \begin{align*}
      H(s)=0.
  \end{align*}

  If $\|u_1\|_{\Upsilon^u_s}> 2m$ and $\|u_2\|_{\Upsilon^u_s}\leq 2m$, then:
  \begin{align*}
    H(s)=&\left\|\theta_m(\|u_2\|_{\Upsilon^u_s})u_2\nabla (G*u_2)\right\|_p\nonumber\\
      =&\left|\theta_m(\|u_1\|_{\Upsilon^u_s})-\theta_m(\|u_2\|_{\Upsilon^u_s})\right|\left\|u_2\nabla (G*u_2)\right\|_p\nonumber\\
      \leq&Cm(\|u_1-u_2\|_{\Upsilon^u_s}).
  \end{align*}

  Similarly, if $\|u_1\|_{\Upsilon^u_s}\leq 2m$ and $\|u_2\|_{\Upsilon^u_s}> 2m$, then:
  \begin{align*}
    H(s)\leq Cm\|u_1-u_2\|_{\Upsilon^u_s}.
  \end{align*}
  Combining the above estimates with (\ref{J1 1}), we get
 \begin{align}\label{J1}
     J_1(t)\leq Cm\|u_1-u_2\|_{\Upsilon^u_t}\int_0^t(t-s)^{-\beta-\frac{1}{2}-\epsilon} \, ds\leq Cm\|u_1-u_2\|_{\Upsilon^u_t}t^{\frac{1}{2}-\epsilon-\beta}.
 \end{align}

To estimate $J_2(t)$, we use inequality (\ref{A1}) to get
\begin{align}\label{J2 1}
    J_2(t) \leq & \int_0^t \left\| e^{-(t-s)A} \left( \theta_m(\|u_1\|_{\Upsilon^u_s}) g(u_1(s)) - \theta_m(\|u_2\|_{\Upsilon^u_s}) g(u_2(s)) \right) \right\|_\infty \, ds \nonumber\\
    \leq & Ct \sup_{s \in [0, t]} \left\| \theta_m(\|u_1\|_{\Upsilon^u_s}) g(u_1(s)) - \theta_m(\|u_2\|_{\Upsilon^u_s}) g(u_2(s)) \right\|_\infty.
\end{align}

Define
$$
G(s) = \left\| \theta_m(\|u_1\|_{\Upsilon^u_s}) g(u_1(s)) - \theta_m(\|u_2\|_{\Upsilon^u_s}) g(u_2(s)) \right\|_\infty.
$$

If $\|u_1\|_{\Upsilon^u_s} \vee \|u_2\|_{\Upsilon^u_s} \leq 2m$, then
\begin{align*}
    G(s) \leq & \left\| \theta_m(\|u_1\|_{\Upsilon^u_s}) g(u_1(s)) - \theta_m(\|u_2\|_{\Upsilon^u_s}) g(u_1(s)) \right\|_\infty \nonumber\\
    & + \left\| \theta_m(\|u_2\|_{\Upsilon^u_s}) g(u_1(s)) - \theta_m(\|u_2\|_{\Upsilon^u_s}) g(u_2(s)) \right\|_\infty \nonumber\\
    \leq & |\theta_m(\|u_1\|_{\Upsilon^u_s}) - \theta_m(\|u_2\|_{\Upsilon^u_s})| \left\| g(u_1(s)) - g(0) \right\|_\infty \nonumber\\
    & + \left\| g(u_1(s)) - g(u_2(s)) \right\|_\infty \nonumber\\
    \leq & \frac{C}{m} \|u_1 - u_2\|_{\Upsilon^u_s} \left(2m \sup_{r \in [-2m, 2m]} |g'(r)| + g(0)\right) \nonumber\\
    & + \sup_{r \in [-2m, 2m]} |g'(r)| \|u_1(s) - u_2(s)\|_\infty \nonumber\\
    \leq & C_m \|u_1 - u_2\|_{\Upsilon^u_s}.
\end{align*}
Here, we have used the fact that $\sup_{r \in [0, \infty)} |\theta_m'(r)| \leq \frac{C}{m}$ and $\sup_{r \in [-2m, 2m]} |g'(r)| \leq C_m$ since $g \in C^1$.

If $\|u_1\|_{\Upsilon^u_s} \wedge \|u_2\|_{\Upsilon^u_s} > 2m$, then
\begin{align*}
    G(s) = 0.
\end{align*}

If $\|u_1\|_{\Upsilon^u_s} > 2m$ and $\|u_2\|_{\Upsilon^u_s} \leq 2m$, then
\begin{align*}
    G(s) = & \left\| \theta_m(\|u_2\|_{\Upsilon^u_s}) g(u_2(s)) \right\|_\infty \nonumber\\
    = & |\theta_m(\|u_1\|_{\Upsilon^u_s}) - \theta_m(\|u_2\|_{\Upsilon^u_s})| \left\| g(u_2(s)) - g(0) \right\|_\infty \nonumber\\
    \leq & C_m \|u_1 - u_2\|_{\Upsilon^u_s}.
\end{align*}

Similarly, if $\|u_1\|_{\Upsilon^u_s} \leq 2m$ and $\|u_2\|_{\Upsilon^u_s} > 2m$, we have
\begin{align*}
    G(s) \leq C_m \|u_1 - u_2\|_{\Upsilon^u_s}.
\end{align*}

Combining the above estimates, and substituting into (\ref{J2 1}), we obtain
\begin{align}\label{J2}
    J_2(t) \leq C_m t \|u_1 - u_2\|_{\Upsilon^u_t}.
\end{align}

For $J_3(t)$, we apply Lemma \ref{M} with $\eta = 1$, $r = \frac{3}{2}$, and $q =\infty$ in (\ref{Mt}) to get
\begin{align}\label{J3}
    \mathbb{E} \sup_{t \in [0, T]} J_3(t) & \leq C (1 + T)^{C'} \mathbb{E}\left(\int_0^T\|\|\sigma(u_1(s)) - \sigma(u_2(s))\|_{\ell^2}\|_{\infty}^3 \right)^{\frac{1}{3}} \nonumber \\
    & \leq C T^\frac{1}{3}(1 + T)^{C'} \mathbb{E} \sup_{s \in [0, T]} \|u_1(s) - u_2(s)\|_\infty,
\end{align}
where $C$ and $C'$ are constants independent of $T$. Combining (\ref{Phi1-2}), (\ref{J1}), (\ref{J2}), and (\ref{J3}), we obtain
\begin{align*}
    \mathbb{E} \sup_{t \in [0, T]} \|\Phi(u_1)(t) - \Phi(u_2)(t)\|_\infty & \leq C_m \mathbb{E} \|u_1 - u_2\|_{\Upsilon^u_T} \left(T^{\frac{1}{2} - \epsilon - \beta} + T^{\frac{1}{3}} (1 + T)^{C'} + T \right),
\end{align*}
with $C_m$ and $C'$ independent of $T$. Therefore, there exist constants $\rho$ and $C_m > 0$ such that
\begin{align*}
    \|\Phi(u_1) - \Phi(u_2)\|_{S_T} & \leq C_m T^\rho \|u_1 - u_2\|_{S_T},
\end{align*}
with $C_m$ independent of $T$.

Choose $T = T_m$ such that $C_m T_m^\rho = \frac{1}{2}$. Then, $\Phi$ is a contraction on the space $S_{T_m}$. Applying the Banach fixed point theorem, we conclude that there exists a unique solution  $u_m $ to equation (\ref{cut}) in the space $u_m \in S_{T_m}$.

Let $S^1_T$ be the space of all $\{\mathcal{F}_t\}_{t \in [0, T]}$-adapted, $C(\bar{\mathcal{O}})$-valued stochastic processes $u(t)$ for $t \geq 0$ such that
$$
\|u\|_{S^1_T}^2 := \mathbb{E} \left( \|u\|_{\Upsilon^u_T}^2 \right) < \infty,
$$
and
$$
u = u_m \text{ on } [0, T_m], \quad \mathbb{P}\text{-a.s.}.
$$
Then, $(S^1_T, \|\cdot\|_{S^1_T})$ is a Banach space.

We introduce a mapping $\Phi^1$ on $S^1_T$ by defining
\begin{align*}
    \Phi^1(u)(t + T_m) := e^{-tA} u_m(T_m) & - \int_{T_m}^{T_m + t} e^{-(T_m + t - s)A} \theta_m(\|u\|_{\Upsilon^u_s}) \left(\nabla \cdot (u \nabla G * u) - g(u)\right) (s) \, ds \\
    & + \int_{T_m}^{T_m + t} e^{-(T_m + t - s)A} \sigma(u(s)) \, dW_s.
\end{align*}
Note that the constant $T_m$ does not depend on the initial datum. By repeating the above arguments, we can solve the cut-off equation (\ref{cut}) for $t \in [T_m, 2T_m]$, $[2T_m, 3T_m]$, \dots, and ultimately obtain a unique solution $u_m \in S_T$ of equation (\ref{cut}) for any $T > 0$.

Define the stopping time
\begin{align}\label{eq stopping time 1}
    \tau_m = \inf \left\{ t > 0 \mid \|u_m\|_{\Upsilon^u_t} \geq m \right\}.
\end{align}
The stopping time $\tau_m$ is well-defined. When $m \gg \|u_0\|_\infty$, we have
$$
\mathbb{P}(\tau_m > 0) = 1.
$$

By the definition of $\theta_m$, it follows that $(u_m(t))_{t \in [0, \tau_m)}$ is a local mild solution to the system (\ref{main}).

To prove the nonnegativity, consider the function $g$ expressed as $g(s) = g(0) + s \tilde{g}(s)$, where $\tilde{g}(s)$ is defined by:
\begin{align*}
    \tilde{g}(s) = \begin{cases}
        g'(0), & \text{if } s = 0, \\
        \frac{g(s) - g(0)}{s}, & \text{if } s > 0.
    \end{cases}
\end{align*}

Define $b(t, x) = -\chi \theta_m(\|u\|_{\Upsilon^u_s}) \nabla G * u(t, x)$, $c(t, x) = \theta_m(\|u\|_{\Upsilon^u_s}) \left( \chi \Delta G * u(t, x) + \tilde{g}(u(t, x)) \right)$, and $\nu(t, x) = I_{\{u(t, x)\not =0\}} \frac{\sigma(u(t, x))}{u(t, x)}$. Then, $u(t, x)$ can be viewed as the solution of the stochastic partial differential equation:
\begin{align*}
    dU = \Delta U \, dt + (b(t) \nabla U + c(t) U + g(0)) \, dt + U \nu(t) \, dW_t,
\end{align*}
where $g(0) \geq 0$. According to \cite[Theorem 5.12]{NK}, it follows that $u(t, x) \geq 0$ almost surely if $u_0(x) \geq 0$.
\hfill $\Box$\end{proof}
\subsection{Uniqueness of Local Solutions}
We recall the following generalized Gronwall-Bellman inequality(Theorem 2.1 and Corollary 2.1 in \cite{W}):

\begin{lem} \label{GGI}
    Let $0 <\alpha< 1$ and consider the time interval $I = [0, T)$, where $T \leq \infty$. Suppose $a(t)$ is a
nonnegative function that is locally integrable on $I$, and $b(t)$ and $g(t)$ are nonnegative, nondecreasing
continuous functions defined on $I$ and bounded by a positive constant $M$. If $f(t)$ is nonnegative
and locally integrable on $I$, and satisfies
\begin{align*}
    f(t)\leq a(t)+b(t)\int_0^tf(s) \, ds+g(t)\int_0^t(t-s)^{\alpha-1}f(s) \, ds,
\end{align*}
then
\begin{align*}
    f(t)\leq a(t)+&\sum_{n=1}^\infty\sum_{i=0}^n\frac{n!(b(t))^{n-i}(g(t)\Gamma(\alpha_1))^i}{i!(n-i)!\Gamma(i\alpha+n-i)}\times\int_0^t(t-s)^{i\alpha+n-i-1}a(s) \, ds,
\end{align*}
where $\Gamma(\cdot)$ is the Gamma function. Moreover, if $a(t)$ is nondecreasing on $I$. Then
\begin{align*}
    f(t)\leq a(t)E_{\alpha}(g(t)\Gamma(\alpha)t^{\alpha})e^{\frac{b(t)t}{\alpha}}.
\end{align*}
Here, the Mittag-Leffler function  $E_{\varrho}(z)$ is  defined by $E_{\varrho}(z)=\sum\limits_{k=0}^\infty \frac{z^k}{\Gamma (k\varrho+1)}$ for $z > 0$, and $E_{\varrho}(z)<\infty$ when  $\varrho>0$.
\end{lem}

The following result gives the uniqueness of local solution of (\ref{main}).

\begin{thm}\label{uniqueness}
Under the assumption ({\bf H-1}), for any $g \in C^{1}([0,\infty))$,  suppose that $(u_1, v_1, \tau^1)$ and $(u_2, v_2, \tau^2)$ are two local mild solutions of the system (\ref{main}). Define $\tau = \tau^1 \wedge \tau^2$. Then we have $(u_1, v_1) = (u_2, v_2)$ on $[0, \tau)$.
\end{thm}

\begin{proof}
Define
$$
\tau_R^i = \inf\{t \geq 0 \mid \|u_i\|_{\Upsilon^u_t} \geq R\} \wedge \tau^i, \quad i = 1, 2,
$$
and set $\tau_R = \tau_R^1 \wedge \tau_R^2$. By the definition of a mild solution, for $t \in [0, T \wedge \tau_R]$ and $\gamma>2$, we have
\begin{align}\label{1-2 R}
    \|u_1(t) - u_2(t)\|_\infty^\gamma \leq & 3^{\gamma-1} \left\|\int_0^t e^{-(t-s)A} \nabla \cdot (u_1 \nabla (G * u_1) - u_2 \nabla (G * u_2))(s) \, ds \right\|_\infty^\gamma \nonumber \\
    & + 3^{\gamma-1} \left\|\int_0^t e^{-(t-s)A} \big(g(u_1(s)) - g(u_2(s))\big) \, ds \right\|_\infty^\gamma \nonumber \\
    & + 3^{\gamma-1} \left\|\int_0^t e^{-(t-s)A} \big(\sigma(u_1(s)) - \sigma(u_2(s))\big) \, dW_s \right\|_\infty^\gamma \nonumber \\
    \leq & 3^{\gamma-1} (J_1(t) + J_2(t) + J_3(t)).
\end{align}

To estimate $J_1(t)$, we proceed similarly as in the proof of (\ref{Phi})
\begin{align}\label{J1R}
    &I_{\{t<\tau_R\}}J_1(t)\nonumber\\
    \leq & C \left( I_{\{t<\tau_R\}}\int_0^{t} \left\|A^\beta e^{-(t-s)A} \nabla \cdot \left(u_1 \nabla (G * u_1) - u_2 \nabla (G * u_2)\right)(s) \right\|_p \, ds \right)^\gamma \nonumber \\
    \leq & C_\epsilon \left(I_{\{t<\tau_R\}} \int_0^{t} (t- s)^{-\beta - \frac{1}{2} - \epsilon} \|\left(u_1 \nabla (G * u_1) - u_1 \nabla (G * u_2) + u_1 \nabla (G * u_2) - u_2 \nabla (G * u_2)\right)(s)\|_p \, ds \right)^\gamma \nonumber \\
    \leq & C_\epsilon \left(I_{\{t<\tau_R\}} \int_0^{t} (t- s)^{\frac{\gamma - 1}{\gamma}(-\beta - \frac{1}{2} - \epsilon)} (t- s)^{\frac{1}{\gamma}(-\beta - \frac{1}{2} - \epsilon)} R \|u_1(s) - u_2(s)\|_\infty \, ds \right)^\gamma \nonumber \\
    \leq & C_{\epsilon, R} I_{\{t<\tau_R\}}\int_0^{t} (t- s)^{-\beta - \frac{1}{2} - \epsilon} \|u_1(s) - u_2(s)\|_\infty^\gamma \, ds \left( \int_0^{t} (t- s)^{-\beta - \frac{1}{2} - \epsilon} \, ds \right)^{\gamma - 1} \nonumber \\
    \leq & C_{\epsilon, R} t^{(\gamma - 1) \left(\frac{1}{2} - \beta - \epsilon \right)} I_{\{t<\tau_R\}}\int_0^{t} (t- s)^{-\beta - \frac{1}{2} - \epsilon} \|u_1(s) - u_2(s)\|_\infty^\gamma \, ds\nonumber \\
    \leq & C_{\epsilon, R} t^{(\gamma - 1) \left(\frac{1}{2} - \beta - \epsilon \right)}\int_0^{t} (t- s)^{-\beta - \frac{1}{2} - \epsilon} I_{\{s<\tau_R\}}\|u_1(s) - u_2(s)\|_\infty^\gamma \, ds.
\end{align}

By H\"older's inequality, we estimate $J_2(t)$ as follows
\begin{align}\label{J2R}
    I_{\{t<\tau_R\}}J_2(t) \leq & \left( \int_0^{t} \|e^{-(t- s)A} \big(g(u_1(s)) - g(u_2(s))\big)\|_\infty \, ds \right)^\gamma \nonumber \\
    \leq & \left( I_{\{t<\tau_R\}}\int_0^{t} \sup_{r \in [-R, R]} |g'(r)| \|u_1(s) - u_2(s)\|_\infty \, ds \right)^\gamma \nonumber \\
    \leq & C_R t^2 \int_0^{t} I_{\{s<\tau_R\}}\|u_1(s) - u_2(s)\|_\infty^\gamma \, ds.
\end{align}

To bound  $J_3(t)$, use Lemma \ref{M} with $\eta = \gamma$, $r = \frac{\gamma}{2}$, and $q = \infty$ in (\ref{Mt}) to get
\begin{align}\label{J3R}
\mathbb{E} \sup_{t \in [0, T]} I_{\{t<\tau_R\}}J_3(t) \leq &\mathbb{E} \sup_{t \in [0, T]}J_3(t \wedge \tau_R)\nonumber\\
    \leq & C \mathbb{E} \int_0^{t \wedge \tau_R} \left\|\left\|\sigma(u_1(s)) - \sigma(u_2(s))\right\|_{\ell^2}\right\|_{\infty}^\gamma \, ds \nonumber \\
    \leq & C \mathbb{E} \int_0^{t}I_{\{s<\tau_R\}} \|u_1(s) - u_2(s)\|_\infty^\gamma \, ds.
\end{align}

Combining (\ref{1-2 R}), (\ref{J1R}), (\ref{J2R}), and (\ref{J3R}), we obtain
\begin{align*}
    &\mathbb{E} \sup_{t \in [0, T]} I_{\{t<\tau_R\}}\|u_1(t) - u_2(t)\|_\infty^\gamma \\
    &\leq  C_{\epsilon, R} T^{(\gamma - 1) \left(\frac{1}{2} - \beta - \epsilon \right)} \mathbb{E}\int_0^{T} (T- s)^{-\beta - \frac{1}{2} - \epsilon} I_{\{s<\tau_R\}}\|u_1(s) - u_2(s)\|_\infty^\gamma \, ds \\
    &\quad + C (T^2 + 1) \mathbb{E} \int_0^{T} I_{\{s<\tau_R\}}\|u_1(s) - u_2(s)\|_\infty^\gamma \, ds.
\end{align*}

Applying Lemma \ref{GGI}, we conclude
\begin{align*}
    \mathbb{E} \sup_{t \in [0, T]} I_{\{t<\tau_R\}}\|u_1(t) - u_2(t)\|_\infty^\gamma = 0.
\end{align*}

We thus obtain uniqueness by noticing that $\tau_R \uparrow \tau$ as $R \uparrow \infty$ and that $T$ is arbitrary.
\hfill $\Box$\end{proof}

\section{Global Existence}\label{section4}
\setcounter{equation}{0}
\begin{defn}\label{def maximal solution}
Let $(u, v, \tau)$ be a local mild solution of the system (\ref{main}). We say that $(u, v, \tau)$ is a maximal local solution if
$$
\limsup_{t \nearrow \tau} \sup_{s \in [0, t]} \left(\|u(s)\|_\infty + \|v(s)\|_{1, \infty}\right) = \infty \text{ on } \{\tau < \infty\} \text{ a.s.}.
$$
\end{defn}

Recall the stopping times $\{\tau_m, m \in \mathbb{N}\}$ defined in (\ref{eq stopping time 1}). By the uniqueness of the local solution established in Section 3.2, we have that $\tau_m \leq \tau_{m+1}$ almost surely, and
$$
(u_{m+1}, v_{m+1}) = (u_m, v_m) \text{ on } [0, \tau_m).
$$

Define a stopping time
$$
\tau = \lim_{m \to \infty} \tau_m,
$$

 and a stochastic process $(u, v)$ on $[0, \tau)$ by
$$
(u, v) = (u_m, v_m) \text{ on } [0, \tau_m).
$$
Then $(u, v, \tau)$ be a local mild solution of the system (\ref{main}).

Since $\|u_m\|_{\Upsilon^u_{\tau_m}} \geq m$ on $\{\tau < \infty\}$, it follows that
$$
\limsup_{t \nearrow \tau} \sup_{s \in [0, t]} \left(\|u(s)\|_\infty + \|v(s)\|_{1, \infty}\right) \geq \lim_{m \to \infty} \|u_m\|_{\Upsilon^u_{\tau_m}} = \infty \text{ on } \{\tau < \infty\}.
$$

Thus, $(u, v, \tau)$ is a maximal local solution of the system (\ref{main}). For the remainder of this section, we assume that $(u, v, \tau)$ is a maximal local solution of the system (\ref{main}).
\vskip 0.3cm
To derive a $L^p$ It\^{o} formula for the solution $u$,
we now introduce the following approximation $U_n$ of $u$:
\begin{align*}
    U_n(t) &= e^{-t A} R_n u_0 - \int_0^t e^{-(t-s) A} \nabla \cdot R_n \left(\chi u(s) \nabla G * u(s)\right) \, ds \nonumber \\
    &\quad + \int_0^t e^{-(t- s) A} R_n g(u(s)) \, ds + \int_0^t e^{-(t-s) A} R_n \sigma(u(s)) \, dW_s,
\end{align*}
where $R_n = n R(n, -A)$, and $R(n, -A) = (nI + A)^{-1}$ denotes the resolvent of $-A$ evaluated at $n \in \mathbb{N}$, with $n \in \rho(-A)$. Here, $\rho(-A)$ represents the resolvent set of $-A$, which is the set of all complex numbers $\lambda$ for which $\lambda I + A$ has a bounded inverse.

The operator $R_n$ is an approximation to the identity operator, known as the Yosida approximation.
This approximation plays a crucial role in regularizing the nonlinear terms in the differential equations.

We also have that $R_n : L^p(\Omega) \to \mathcal{D}(A)$ and according to Theorem 1.4 from \cite{GM} and the condition (\ref{A1}), it follows that $n \in \rho(-A)$ for every $n > \nu_1$. Moreover, for any $u \in L^p(\mathcal{O})$,
$$
\|R(n, -A) u\|_p \leq C \frac{\|u\|_p}{n - \nu_1},
$$
where $C$ is a constant independent of $n$. For further details on Yosida approximations, see Chapter 1.2 of \cite{GM}.

Regarding the process the process $U_n(t)$, we have the following conclusions.

\begin{prop}\label{prop}
    For every $m\geq 1$, the process  $U_n(t)$ satisfies the following equation in $L^2(\mathcal{O})$ for $t< \tau_m$,
    \begin{align}\label{un-1}
        U_n(t) = \int_0^{t} \left( -A U_n(s) - \nabla \cdot R_n \left( \chi u(s) \nabla G * u(s) \right) + R_n g(u(s)) \right) \, ds
        + \int_0^{t} R_n \sigma(u(s)) \, dW_s.
    \end{align}
\end{prop}

\begin{proof}
    Regularized by the resolvent operator $R_n$, it is easy to see that $U_n(t)\in D(A)$ almost surely. It remains to show that the equation (\ref{un-1}) holds. From the definition of $U_n(t)$  it follows that
    \begin{align*}
        \int_0^t  A U_n(s) \, ds
        &= \int_0^t  e^{-sA} A R_n u_0 \, ds- \int_0^t  \int_0^s A e^{-(s-r)A} \nabla \cdot R_n \left( \chi u(r) \nabla G * u(r) \right) \, dr \, ds \nonumber \\
        &\quad + \int_0^t  \int_0^s A e^{-(s-r)A} R_n g(u(r)) \, dr \, ds + \int_0^t  \int_0^s Ae^{-(s-r) A} R_n \sigma(u(r)) \, dW_r \, ds.
    \end{align*}
Applying Fubini's theorem and the stochastic Fubini theorem (see Theorem 2.8 in \cite{GM} and Appendix A of \cite{MR})  we get
    \begin{align}\label{un-2}
        \int_0^t A U_n(s) \, ds
        &= \int_0^t e^{-sA} A R_n u_0 \, ds - \int_0^t \int_r^t  e^{-(s-r)A} A \nabla \cdot R_n \left( \chi u(r) \nabla G * u(r) \right) \, ds \, dr \nonumber \\
        &\quad + \int_0^t \int_r^t e^{-(s-r)A} A R_n g(u(r)) \, ds \, dr+ \int_0^t \int_r^t e^{-(s-r)A} A R_n \sigma(u(r)) \, ds \, dW_r.
    \end{align}
Invoking the identity
    \begin{align*}
        \int_0^t e^{-sA} A \phi \, ds =e^{-tA} \phi - \phi
    \end{align*}
    for $\phi \in \mathcal{D}(A)$, (\ref{un-2}) easily yields that
    \begin{align*}
        \int_0^t A U_n(s) \, ds = U_n(t) + \int_0^{t} \left( \nabla \cdot R_n \left( \chi u(s) \nabla G * u(s) \right) - R_n g(u(s)) \right) \, ds  - \int_0^{t} R_n \sigma(u(s)) \, dW_s.
    \end{align*}
\hfill $\Box$\end{proof}

The next result provides the convergence of $U_n(t)$ to $u(t)$.

\begin{lem}\label{lem 4.7}
    For all $m \in \mathbb{N}$ and $T \geq 0$, we have
    $$u \in L^\infty([0, T \wedge \tau_m]; C (\bar{\mathcal{O}})) \cap L^2([0, T \wedge \tau_m]; W^{1,2}(\mathcal{O}))$$ almost surely,

    \begin{align}\label{W12}
        \lim_{n \rightarrow \infty} \mathbb{E} \left( \int_0^{T \wedge \tau_m} \| \nabla U_n(t) - \nabla u(t) \|_2^2dt \right) = 0,
    \end{align}

    and
    \begin{align}\label{infty}
        \lim_{n \rightarrow \infty} \mathbb{E} \left( \sup_{t \in [0, T \wedge \tau_m]} \| U_n(t) - u(t) \|_\infty \right) = 0.
    \end{align}
\end{lem}

\begin{proof}
     It was already proved that $u\in L^\infty([0, T \wedge \tau_m]; C (\bar{\mathcal{O}}))$. We now show that $u \in L^2([0, T \wedge \tau_m]; W^{1,2}(\mathcal{O}))$.
     Recall
    \begin{align*}
     u(t)=& e^{-tA}u_0 - \int_0^te^{-(t-s)A} \left(\nabla \cdot \left(\chi u(s) \nabla G * u(s)\right)\right) \, ds \nonumber \\
     & - \int_0^t e^{-(t-s)A} g(u(s)) \, ds + \int_0^t e^{-(t-s)A} \sigma(u(s)) \, dW_s \nonumber\\
        :=&I_1(t)+I_2(t)+I_3(t)+I_4(t).
     \end{align*}
    To prove that $E[\int_0^{T \wedge \tau_m}\|\nabla u(t)\|_2^2dt]<\infty$, it suffices to show that $E[\int_0^{T \wedge \tau_m}\|\nabla I_i(t)\|_2^2dt]<\infty$, $i=1,2,3,4$.
   Recalling the following regularizing property of the semigroup $e^{-tA}$ stated in Lemma 2.1
   \begin{equation}\label{4.7-1}
      \|\nabla e^{-tA}f\|_{2}\leq C(1+t^{-\frac{1}{2}})e^{-\nu_1 t}\| f\|_2, \quad f\in L^{p}(\mathcal{O}),
  \end{equation}
  using the fact that $\sup_{0\leq t\leq T \wedge \tau_m}\|u(t)\|_{\infty}\leq m$, one can easily show that $E[\int_0^{T \wedge \tau_m}\|\nabla I_i(t)\|_2^2dt]<\infty$ for $i=3,4$.

By the chain rule, we have
\begin{align*}
    2\int_0^T \|\nabla e^{-tA} u_0\|_2^2 \, dt = \|u_0\|_2^2-\|e^{-TA} u_0\|_2^2.
\end{align*}
This shows that $E[\int_0^{T \wedge \tau_m}\|\nabla I_1(t)\|_2^2dt] \leq \|u_0\|_2^2<\infty$.

For the term $I_2(t)$, we notice that for $I_2(t)$ is the variational solution of the following partial differential equation:
  $$\frac{\partial I_2(t)}{\partial t}=-AI_2(t)-\nabla \cdot \left(\chi u(t) \nabla G * u(t)\right).$$
  Use the chain rule, integration by parts and the Young's inequality to get
  \begin{align*}
     \|I_2(t)\|_2^2=& -2\int_0^t\|\nabla I_2(s)\|_2^2ds - 2\int_0^t<I_2(s), \left(\nabla \cdot \left(\chi u(s) \nabla G * u(s)\right)\right)>_{L^{2}(\mathcal{O}} \, ds \nonumber \\
    = & -2\int_0^t\|\nabla I_2(s)\|_2^2ds +2\int_0^t<\nabla I_2(s), \left(\chi u(s) \nabla G * u(s)\right)>_{L^{2}(\mathcal{O}} \, ds \nonumber \\
    \leq & -\int_0^t\|\nabla I_2(s)\|_2^2ds +\int_0^t\left\|\chi u(s) \nabla G * u(s)\right)\|_2^2\, ds.
     \end{align*}
  Moving the first term on right side to the left side we obtain
  $$E[\int_0^{T \wedge \tau_m}\|\nabla I_2(t)\|_2^2dt]\leq E[\int_0^{T \wedge \tau_m}\left\|\chi u(t) \nabla G * u(t)\right)\|_2^2\, dt]\leq D_m,$$
  $D_m$ is a constant depending on $m$.
  Putting the above arguments together, we see that $u\in L^2([0, T \wedge \tau_m]; W^{1,2}(\mathcal{O}))$ and moreover,
  $$E[\int_0^{T \wedge \tau_m}\|\nabla u(t)\|_2^2dt]\leq C_m.$$

Now we turn to the proof of the convergence of $U_n(t)$, namely (\ref{W12}) and (\ref{infty}).
By definition of $u(t)$ and $U_n(t)$, we have
\begin{align*}
     U_n(t)-u(t)=& e^{-tA} (R_n - I) u_0 - \int_0^t e^{-(t-s)A} \nabla \cdot (R_n - I) \left( \chi u(s) \nabla G * u(s) \right) ds \nonumber \\
     & - \int_0^t e^{-(t-s)A} (R_n - I) g(u(s)) \, ds + \int_0^t e^{-(t-s)A} (R_n - I)\sigma(u(s)) \, dW_s \nonumber\\
        :=&J_1^n(t)+J_2^n(t)+J_3^n(t)+J_4^n(t).
     \end{align*}
     To prove (\ref{W12}), it suffices to show that $\lim_{n \rightarrow \infty}E[\int_0^{T \wedge \tau_m}\|\nabla J_i^n(t)\|_2^2dt]=0$, $i=1,2,3,4$. Using (\ref{4.7-1}), we can easily show that
     \begin{align*}
         E[\int_0^{T \wedge \tau_m}\|\nabla J_3^n(t)\|_2^2dt]\leq C E[\int_0^{T \wedge \tau_m} \|(R_n - I) g(u(s))\|_2^2dt],
     \end{align*}
     and
     \begin{align*}
         E[\int_0^{T \wedge \tau_m}\|\nabla J_4^n(t)\|_2^2dt]\leq C E[\int_0^{T \wedge \tau_m} \|\|(R_n - I) \sigma(u(s))\|_{l^2}\|_2^2dt].
     \end{align*}

Similarly as the proofs of $I_1(t)$ and $I_2(t)$, we can show the following inequalities:
 \begin{align*}
         E[\int_0^{T \wedge \tau_m}\|\nabla J_1^n(t)\|_2^2dt]\leq C\| (R_n - I)u_0\|_2^2,
     \end{align*}
     and
     \begin{align*}
         E[\int_0^{T \wedge \tau_m}\|\nabla J_2^n(t)\|_2^2dt]\leq E[\int_0^{T \wedge \tau_m}\left\|\chi (R_n - I)u(t) \nabla G * u(t)\right\|_2^2\, dt].
     \end{align*}
     Since $\| (R_n - I)u_0\|_2$, $\|\chi (R_n - I)u(t) \nabla G * u(t)\|_2 $, $\|(R_n - I) g(u(s))\|_2$, and $\|\|(R_n - I) \sigma(u(s)) \|_{\ell^2} \|_{2}$ tend to 0 as $n \to \infty$, and are bounded by $C(m^2 + m + 1)$ for $s \in [0, \tau_m]$, by the H\"older inequality and the Dominated Convergence Theorem we deduce that   $\lim_{n \rightarrow \infty}E[\int_0^{T \wedge \tau_m}\|\nabla J_i^n(t)\|_2^2dt]=0$, $i=1,2,3,4$. Thus, we complete the proof of (\ref{W12}).

Now, we prove (\ref{infty}), for $t \in [0, \tau_m]$, we have
\begin{align}\label{4.30}
    \|U_n(t) - u(t)\|_\infty
    &\leq  \| e^{-tA} (R_n - I) u_0 \|_\infty \nonumber \\
    &\quad +  \left\| \int_0^t e^{-(t-s)A} \nabla \cdot (R_n - I) \left( \chi u(s) \nabla G * u(s) \right) ds \right\|_\infty \nonumber \\
    &\quad + \int_0^t \| e^{-(t-s)A} (R_n - I) g(u(s)) \|_\infty \, ds \nonumber \\
    &\quad + \left\| \int_0^t e^{-(t-s)A} (R_n - I) \sigma(u(s)) \, dW_s \right\|_\infty \nonumber \\
    &\leq C \| (R_n - I) u_0 \|_\infty \nonumber \\
    &\quad + C_\epsilon \int_0^t (t-s)^{-\beta - \frac{1}{2} - \epsilon} \| (R_n - I) \left( \chi u(s) \nabla G * u(s) \right) \|_p \, ds \nonumber \\
    &\quad + \int_0^t \| (R_n - I) g(u(s)) \|_\infty \, ds \nonumber \\
    &\quad + \left\| \int_0^t e^{-(t-s)A} (R_n - I) \sigma(u(s)) \, dW_s \right\|_\infty.
\end{align}
Where, the regularizing property (\ref{A4}) has been used, $p$ is sufficiently big and $\beta>\frac{1}{p}$.
Applying Lemma \ref{M} with $\eta = 1$, $\frac{1}{r}+\frac{1}{q}<1$ in (\ref{Mt}), we have
\begin{align*}
\mathbb{E} \sup_{t \in [0, T \wedge \tau_m]}
\left\| \int_0^t e^{-(t-s)A} (R_n - I) \sigma(u(s)) \, dW_s \right\|_{\infty}
&\leq C_T \mathbb{E} \left( \int_{0}^{T\wedge \tau_m} \left\| \left\|(R_n - I) \sigma(u(s)) \right\|_{\ell^2} \right\|_{2q}^{2r} \, ds \right)^{\frac{1}{2r}}
\end{align*}

Since $\|(R_n - I) u_0\|_\infty$, $\|(R_n - I) \left( \chi u(s) \nabla G * u(s) \right)\|_{p}$, $\|(R_n - I) g(u(s))\|_\infty$, and $\|(R_n - I) \sigma(u(s)) \|_{\ell^2} \|_{2q}$ tend to 0 as $n \to \infty$, and are bounded by $C(m^2 + m + 1)$ (by assumption (H)) for $s \in [0, \tau_m]$, by the H\"older inequality and the Dominated Convergence Theorem we deduce that
\begin{align*}
&\mathbb{E} \sup_{t \in [0, T \wedge \tau_m]}\int_0^t (t-s)^{-\beta - \frac{1}{2} - \epsilon} \|(R_n - I) \left( \chi u(s) \nabla G * u(s) \right) \|_{p} \, ds\rightarrow 0,\\
&\mathbb{E}[\int_0^{T \wedge \tau_m} \|(R_n - I) g(u(s)) \|_\infty \, ds]\rightarrow 0,\\
&
\mathbb{E}[\int_0^{T \wedge \tau_m}\left\| \left\|(R_n - I) \sigma(u(s)) \right\|_{\ell^2} \right\|_{2q}^{2r} \, ds]\rightarrow 0,
\end{align*}
as $n \to \infty$. Consequently, it follows from (\ref{4.30}) that
$$\lim_{n \rightarrow \infty} \mathbb{E} \left( \sup_{t \in [0, T \wedge \tau_m]} \| U_n(t) - u(t) \|_\infty \right) = 0.$$
\hfill $\Box$\end{proof}

Now, we will derive an It\^{o} formula for
$\|u(t)\|_{p}^{p}$ using the approximation $U_n(t)$.

\begin{thm}\label{thm 4.8}
    Let $(u, v, \tau_m)$ be a local mild solution of system (\ref{main}), with $\tau_m$ defined by
    \begin{align*}
        \tau_m = \inf \{ t > 0 \mid \|u\|_{\Upsilon^u_t} \geq m \}.
    \end{align*}
    Then, for $p\geq 2$, the following It\^{o} formula for $\|u(t)\|_{p}^{p}$ holds $\mathbb{P}$-a.s. for all $t \in [0, T]$,
    \begin{align*}
        &\|u(t \wedge \tau_m)\|_{p}^{p} - \|u_0\|_{p}^{p} +p(p - 1)  \int_0^{t\wedge\tau_m}\|\nabla u^{\frac{p}{2}}(s)\|_2^2 \, ds \nonumber \\
        &= p(p - 1) \chi \int_0^{t \wedge \tau_m} \int_\mathcal{O} u^{p - 1}(s) \nabla u(s) \cdot \nabla G * u(s) \, dx \, ds \nonumber \\
        &\quad + p \int_0^{t \wedge \tau_m} \int_\mathcal{O} u^{p - 1}(s) g(u(s)) \, dx \, ds \nonumber \\
        &\quad + p \sum_{k=1}^\infty \int_0^{t \wedge \tau_m} \int_\mathcal{O} \sigma_k(u(s)) u(s) |u(s)|^{p - 2} \, dx \, dW^k_s \nonumber \\
        &\quad + \frac{p(p - 1)}{2} \int_0^{t \wedge \tau_m} \int_\mathcal{O} \sum_{k=1}^\infty |\sigma_k(u(s))|^2 |u(s)|^{p - 2} \, dx \, ds.
    \end{align*}
\end{thm}

\begin{proof}
    In view of Proposition \ref{prop}, the usual It\^{o} formula (see, for example, \cite{DZ}) holds $\mathbb{P}$-a.s. for all $t \in [0, \tau_m]$:
    \begin{align*}
        &F(U_n(t)) - F(u_0) = \int_0^t F'(U_n(s)) \left( -A U_n(s) - \nabla \cdot R_n (\chi u(s) \nabla G * u(s) + R_n g(u(s))) \right)  \, ds \nonumber \\
        &\quad + \sum_{k=1}^\infty \int_0^t  F'(U_n(s)) \sigma_k(u(s))  \, dW^k_s  + \int_0^t  \sum_{k=1}^\infty |\sigma_k(u(s))|^2 F''(U_n(s))  \, ds.
    \end{align*}
    Here, $F : \mathbb{R} \rightarrow \mathbb{R}$ has continuous partial derivatives $F'$ and $F''$. Applying the methodology outlined in Lemma 2 of \cite{DG}, and using the following approximation of $\|U_n(s)\|_{p}^p$:
    \begin{align*}
        \int_{\mathcal{O}}F_n(U_n(t))dx
    \end{align*}
    with $F_n$ given by
    \begin{align*}
        F_n(r) := \begin{cases}
            |r|^p & \text{if } |r| < n, \\
            n^{p-2} \frac{p(p-1)}{2} (|r| - n)^2 + p n^{p-1} (|r| - n) + n^p & \text{if } |r| \geq n,
        \end{cases}
    \end{align*}
    we obtain, $\mathbb{P}$-a.s., for all $t \in [0, T \wedge \tau_m]$:
    \begin{align}\label{4.35}
        &\|U_n(t)\|_{p}^{p} - \|u_0\|_{p}^{p} + p(p - 1) \int_0^t \|\nabla U_n^{\frac{p}{2}}(s)\|_2^2 \, ds \nonumber \\
        &= - p \chi \int_0^t \int_\mathcal{O} (U_n(s))^{p - 1} \nabla \cdot R_n (u(s) \nabla G * u(s)) \, dx \, ds \nonumber \\
        &\quad + p \int_0^t \int_\mathcal{O} (U_n(s))^{p - 1} R_ng(u(s)) \, dx \, ds \nonumber \\
        &\quad + p \sum_{k=1}^\infty \int_0^t \int_\mathcal{O} \sigma_k(u(s)) U_n(s) |U_n(s)|^{p - 2} \, dx \, dW^k_s \nonumber \\
        &\quad + \frac{p(p - 1)}{2} \int_0^t \int_\mathcal{O} \sum_{k=1}^\infty |\sigma_k(u(s))|^2 |U_n(s)|^{p - 2} \, dx \, ds.
    \end{align}

%
Replacing $t$ by $t\wedge \tau_m$ in (\ref{4.35}), we will show that each term  of the equation (\ref{4.35}) converges as $n\rightarrow \infty$.

By Lemma \ref{lem 4.7}, taking a subsequence if necessary we have
\begin{align*}
\|U_n(t \wedge \tau_m)\|_p^p \to \|u(t \wedge \tau_m)\|_p^p
\end{align*}
almost surely.

Again by Lemma \ref{lem 4.7}, Dominated Convergence Theorem and Ito isometry it is easy to see that

\begin{align*}
    \int_0^{t \wedge \tau_m} \|\nabla U_n^{\frac{p}{2}}(s)\|_2^2 \, ds \to \int_0^{t \wedge \tau_m} \|\nabla (u(s))^{\frac{p}{2}}\|_2^2 \, ds,
\end{align*}

\begin{align}
    \int_0^{t \wedge \tau_m} \int_\mathcal{O} (U_n(s))^{p-1} R_ng(u(s)) \, dx \, ds \to \int_0^{t \wedge \tau_m} \int_\mathcal{O} u^{p-1}(s) g(u(s)) \, dx \, ds, \nonumber
\end{align}

\begin{align}
    \int_0^{t \wedge \tau_m} \int_\mathcal{O} \sum_{k=1}^\infty |\sigma_k(u(s))|^2 |U_n(s)|^{p-2} \, dx \, ds \to \int_0^{t \wedge \tau_m} \int_\mathcal{O} \sum_{k=1}^\infty |\sigma_k(u(s))|^2 |u(s)|^{p-2} \, dx \, ds,\nonumber
\end{align}
and
\begin{align*}
    \sum_{k=1}^\infty \int_0^{t \wedge \tau_m} \int_\mathcal{O} \sigma_k(u(s)) U_n(s) |U_n(s)|^{p-2} \, dx \, dW^k_s \to \sum_{k=1}^\infty \int_0^{t \wedge \tau_m} \int_\mathcal{O} \sigma_k(u(s)) u(s) |u(s)|^{p-2} \, dx \, dW^k_s.
\end{align*}

We now consider the first term on the right-hand side of (\ref{4.35}). We have
\begin{align}\label{4.35-1}
    & \left |\int_0^{t \wedge \tau_m} \int_\mathcal{O} (U_n(s))^{p-1} \nabla \cdot R_n \left(u(s) \nabla G * u(s) - u^{p-1}(s) \nabla \cdot u(s) \nabla G * u(s)\right) \, dx \, ds\right | \nonumber \\
    &\leq \left | \int_0^{t \wedge \tau_m} \int_\mathcal{O} (U_n(s))^{p-1} \nabla \cdot (R_n - I) \left(u(s) \nabla G * u(s)\right) \, dx \, ds\right | \nonumber\\
    &\quad + \left |\int_0^{t \wedge \tau_m} \int_\mathcal{O} \left((U_n(s))^{p-1} - u^{p-1}(s)\right) \nabla \cdot \left(u(s) \nabla G * u(s)\right) \, dx \, ds \right |\nonumber\\
    &\leq \sup_{s \in [0, t \wedge \tau_m]} \|U_n(s))\|_{\infty}^{p-1} \int_0^{t \wedge \tau_m} \int_\mathcal{O}| \nabla \cdot (R_n - I) \left(u(s) \nabla G * u(s)\right)| \, dx \, ds \nonumber\\
    &\quad + C\sup_{s \in [0, t \wedge \tau_m]} \|(U_n(s))^{p-1} - u^{p-1}(s)\|_{\infty} \left(\int_0^{t \wedge \tau_m} \int_\mathcal{O}\left| \nabla \cdot \left(u(s) \nabla G * u(s)\right)\right|^2 \, dx \, ds\right)^{\frac{1}{2}}.
\end{align}
Note that, by Lemma \ref{lem 4.7},
 $$\nabla \cdot  \left(u(s) \nabla G * u(s)\right) \in  L^2([0, T \wedge \tau_m]; L^2(\mathcal{O})).$$
Since  $e^{-tA}$ is strongly continuous on the space $W^{1,2}(\mathcal{O})$, $R_n\left(u(s) \nabla G * u(s)\right)\rightarrow u(s) \nabla G * u(s)$ in $L^2([0, T \wedge \tau_m]; W^{1,2}(\mathcal{O}))$.
Therefore, it follows from (\ref{4.35-1}) that
\begin{align*}
    & \int_0^{t \wedge \tau_m} \int_\mathcal{O} (U_n)^{p-1}(s) \nabla \cdot R_n \left(u(s) \nabla G * u(s)\right) \, dx \, ds \to \int_0^{t \wedge \tau_m} \int_\mathcal{O} u^{p-1}(s) \nabla \cdot \left(u(s) \nabla G * u(s)\right) \, dx \, ds,
\end{align*}
as $n\rightarrow \infty$.

To complete the proof, it remains to replace $t$ by $t\wedge \tau_m$ and let $n\rightarrow \infty$ in (\ref{4.35}).
\hfill $\Box$\end{proof}

\vskip 0.3cm

The critical step towards proving the global existence of equation (\ref{main}) is to get  the following a priori bound. This bound asserts that for some $p_0$ greater than $2$, the solution $u(t)$ is uniformly bounded in $L^{p_0}$.
\begin{lem}\label{lem gamma}
Assuming $g$ satisfies (\ref{g}) with $\mu > \frac{\chi}{2}$, for $p_0 \in (2, \frac{\chi}{(\chi-\mu)^+})$, there exists a constant $C_{p_0,\|u_0\|_\infty} > 0$ such that for all $T > 0$, the following inequality holds:
\begin{equation*}
\mathbb{E}\sup_{t\in  [0,T\wedge\tau]}\|u(t)\|_{p_0}^{p_0} \leq C_{{p_0},\|u_0\|_\infty}.
\end{equation*}
\end{lem}

\begin{proof}
For $m\geq 1$, we apply Theorem \ref{thm 4.8} to obtain
\begin{align}\label{lem 1}
&\|u(t\wedge\tau_m)\|_{p_0}^{p_0} - \|u_0\|_{p_0}^{p_0} + {p_0}({p_0} - 1) \int_0^{t\wedge\tau_m} \|\nabla u^{\frac{p}{2}}(s)\|_2^2 \, ds\nonumber\\
&= {p_0}({p_0}-1)\chi\int_0^{t\wedge\tau_m}\int_\mathcal{O} u^{{p_0}-1}(s)\nabla u(s)\cdot \nabla G*u(s) \, dx \, ds \nonumber \\
&\quad + {p_0}\int_{0}^{t\wedge\tau_m}\int_\mathcal{O} u^{{p_0}-1}(s)g(u(s)) \, dx \, ds \nonumber \\
&\quad+ {p_0}\sum_{k=1}^\infty\int_0^{t\wedge\tau_m}\int_\mathcal{O}\sigma_k(u(s))u(s)|u(s)|^{{p_0}-2} \, dx \, dW^k_s \nonumber \\
&\quad+ \frac{{p_0}({p_0}-1)}{2}\int_{0}^{t\wedge\tau_m}\int_\mathcal{O}\sum_{k=1}^\infty|\sigma_k(u(s))|^2|u(s)|^{{p_0}-2} \, dx \, ds.
\end{align}

Multiplying the second equation in (\ref{main}) by $u^{p_0}$, we obtain
\begin{align}\label{lem 2}
{p_0}\int_\mathcal{O} u^{{p_0}-1}(t)\nabla u\cdot\nabla G*u(t) \, dx = -\int_\mathcal{O} u^{p_0}(t) G*u(t) \, dx + \int_\mathcal{O} u^{{p_0}+1}(t) \, dx.
\end{align}

Substituting (\ref{lem 2}) into (\ref{lem 1}), keeping in mind that $u(t)$ is non-negative and using the condition (\ref{g}), we have
\begin{align}\label{lem 3}
&\|u({t\wedge\tau_m})\|_{p_0}^{p_0} - \|u(0)\|_{p_0}^{p_0} + {p_0}({p_0} - 1) \int_0^{t\wedge\tau_m} \|\nabla u^{\frac{p}{2}}(s)\|_2^2 \, ds\nonumber\\
&= -({p_0}-1)\chi\int_0^{t\wedge\tau_m} \int_\mathcal{O} u^{p_0}(s) G*u(s) \, dx \, ds\nonumber \\
&\quad+ ({p_0}-1)\chi \int_{0}^{t\wedge\tau_m} \int_\mathcal{O} u^{{p_0}+1}(s) \, dx \, ds \nonumber \\
&\quad+ {p_0}\int_{0}^{t\wedge\tau_m}\int_\mathcal{O} u^{{p_0}-1}(s)g(u(s)) \, dx \, ds \nonumber \\
&\quad+ {p_0}\sum_{k=1}^\infty\int_0^{t\wedge\tau_m}\int_\mathcal{O}\sigma_k(u(s))u(s)|u(s)|^{{p_0}-2} \, dx \, dW^k_s \nonumber \\
&\quad+ \frac{{p_0}({p_0}-1)}{2}\int_{0}^{t\wedge\tau_m}\int_\mathcal{O}\|\sigma(u(s))\|_{\ell^2}^2|u(s)|^{{p_0}-2} \, dx \, ds \nonumber \\
&\leq -({p_0}\mu-({p_0}-1)\chi)\int_0^{t\wedge\tau_m} \int_\mathcal{O} u^{{p_0}+1}(s) \, dx \, ds + c_1{p_0}\int_{0}^{t\wedge\tau_m}\int_\mathcal{O} u^{{p_0}-1}(s) \, dx \, ds \nonumber \\
&\quad+ {p_0}\sum_{k=1}^\infty\int_0^{t\wedge\tau_m}\int_\mathcal{O}\sigma_k(u(s))u(s)|u(s)|^{{p_0}-2} \, dx \, dW^k_s \nonumber \\
&\quad+ \frac{{p_0}({p_0}-1)}{2}\int_{0}^{t\wedge\tau_m}\int_\mathcal{O}\|\sigma(u(s))\|_{\ell^2}^2|u(s)|^{{p_0}-2} \, dx \, ds.
\end{align}

Since $\delta = {p_0}\mu-({p_0}-1)\chi > 0$ for $\mu > \frac{\chi}{2}$ and ${p_0} \in (2, \frac{\chi}{(\chi-\mu)^+})$, we can apply Young's inequality to bound
\begin{align*}
c_1{p_0}\int_\mathcal{O} u^{{p_0}-1}(s) \, dx \leq \frac{\delta}{2}\int_\mathcal{O} u^{{p_0}+1}(s) \, dx + C,
\end{align*}
and the BDG inequality along with Young's inequality yields
\begin{align*}
&\mathbb{E}\sup_{t\in[0,T]}\sum_{k=1}^\infty\int_0^{t\wedge\tau_m}\int_\mathcal{O}\sigma_k(u(s))u(s)|u(s)|^{{p_0}-2} \, dx \, dW^k_s\\
&\leq C\mathbb{E}\left(\int_0^{T\wedge\tau_m}\sum_{k=1}^\infty\left(\int_\mathcal{O}\sigma_k(u(s))u(s)|u(s)|^{{p_0}-2} \, dx\right)^2 \, ds\right)^{1/2} \\
&\leq C\mathbb{E}\left(\sup_{s\in[0,T\wedge\tau_m]}\int_\mathcal{O}|u(s)|^{{p_0}} \, dx\right)^{1/2} \left(\int_0^{T\wedge\tau_m}\int_\mathcal{O}\|\sigma(u(s))\|_{\ell^2}^2|u(s)|^{{p_0}-2} \, dx \, ds\right)^{1/2} \\
&\leq \frac{1}{2}\mathbb{E}\sup_{t\in[0,T]}\|u({t\wedge\tau_m})\|_{p_0}^{p_0} + C\mathbb{E}\left(\int_0^{T\wedge\tau_m}\int_\mathcal{O}\|\sigma(u(s))\|_{\ell^2}^2|u(s)|^{{p_0}-2} \, dx \, ds\right).
\end{align*}

Taking the supremum over $t$ and the expectation in (\ref{lem 3}), and incorporating the above inequalities, we deduce
\begin{align*}
&\mathbb{E}\sup_{t\in[0,T]}\|u({t\wedge\tau_m})\|_{p_0}^{p_0} + \frac{\delta}{2}\mathbb{E}\sup_{t\in[0,T]}\int_\mathcal{O} u^{{p_0}+1}(t) \, dx
\\
&\leq \|u(0)\|_{p_0}^{p_0} + C + C\mathbb{E}\left(\int_0^{T\wedge\tau_m}\int_\mathcal{O}\|\sigma(u(s))\|_{\ell^2}^2|u(s)|^{{p_0}-2} \, dx \, ds\right) \\
&\leq \|u(0)\|_{p_0}^{p_0} + C + C\mathbb{E}\left(\int_0^{T\wedge\tau_m}\int_\mathcal{O}|u(s)|^{{p_0}} + 1 \, dx \, ds\right).
\end{align*}

Applying the Gronwall inequality leads to
\begin{align*}
\mathbb{E}\sup_{t\in[0,T]}\|u({t\wedge\tau_m})\|_{p_0}^{p_0} \leq C_{{p_0},\|u_0\|_\infty}.
\end{align*}
Since the constant on the right-hand side is independent of $m$, we let $m$ tend to infinity to complete the proof.
\hfill $\Box$\end{proof}

\begin{lem}\label{lem 4.10}
Suppose $g$ satisfies condition (\ref{g}) with $\mu>\frac{\chi}{2}$. Then, there exists a constant $C_{\|u_0\|_\infty}>0$ such that for all $T>0$, we have
\begin{align*}
    \mathbb{E}\sup_{t\in  [0,T\wedge\tau]}\|u(t)\|_\infty \leq C_{\|u_0\|_\infty}.
\end{align*}
\end{lem}

\begin{proof}
Let  $p_0\in (2,\frac{\chi}{(\chi-\mu)^+})$.
Since $u$ is non-negative, we have $\|u(t)\|_\infty=\sup_{x\in \mathcal{O}}u(t,x)$. Keeping this in mind and using the assumptions on the function $g$,  similar to the proof of  equation (\ref{Phi}), we can derive that for $t<\tau$,
\begin{align}\label{4 I}
    \|u(t)\|_\infty \leq & \|e^{-tA}u_0\|_\infty + C \int_0^{t}\|(A_{p_0}+1)^\beta e^{-(t-s)A}\nabla\cdot(u\nabla G*u)(s)\|_{p_0} \, ds  \nonumber\\
    & + \sup_{x\in \mathcal{O}}(\int_0^{t}e^{-(t-s)A}g(u(s)) \, ds )+ \|\int_0^{t} e^{-(t-s)A}\sigma(u(s)) \, dW_s\|_{\infty}
    \nonumber\\
    \leq & C\|u_0\|_\infty + I_1(t) + C_T + I_2(t).
\end{align}

Using condition (\ref{A2}), for $\epsilon\in (0,\frac{1}{2}-\beta)$, we estimate $I_1(t)$ as follows:
\begin{align*}
    I_1(t) \leq & C_{\epsilon}\int_0^{t} (t-s)^{-\beta-\frac{1}{2}-\epsilon} \|(u\nabla G*u)(s)\|_{p_0} \, ds  \nonumber\\
    \leq & C_{\epsilon}\int_0^{t} (t-s)^{-\beta-\frac{1}{2}-\epsilon} \|u(s)\|_{p_0}^2 \, ds  \nonumber\\
    \leq & C \sup_{s\in[0,{t}]} \|u(s)\|_{p_0}^{2} t^{\frac{1}{2}-\beta-\epsilon}.
\end{align*}


To estimate $I_2$, we apply Lemma \ref{M}, choosing $\eta=1$,$r=\infty$, and $q=\frac{p_0}{2}$ in (\ref{Mt}), to get
\begin{align*}
    \mathbb{E}\sup_{t\in[0,T\wedge \tau]} I_2(t) \leq C \mathbb{E}\sup_{t\in[0,T\wedge\tau]} \|\|\sigma(u(t))\|_{\ell^2}\|_{p_0} \leq C \left( \mathbb{E}\sup_{t\in[0,T\wedge\tau]} \|u(s)\|_{p_0} + 1 \right).
\end{align*}

Now, taking the supremum up to $ T\wedge\tau$ and taking the expectation in (\ref{4 I}), and applying the estimates we obtained for $I_1(t)$, $I_2(t)$, we deduce that
\begin{align*}
    \mathbb{E}\sup_{t\in  [0,T\wedge\tau]}\|u(t)\|_\infty &\leq C\mathbb{E}\sup_{t\in  [0,T\wedge\tau]}\|u(t)\|_{p_0}^{2}+C\mathbb{E}\sup_{t\in  [0,T\wedge\tau]}\|u(t)\|_{p_0} \\
    &\leq C\left(\mathbb{E}\sup_{t\in  [0,T\wedge\tau]}\|u(t)\|_{p_0}^{p_0}\right)^{\frac{2}{p_0}} +C\left(\mathbb{E}\sup_{t\in  [0,T\wedge\tau]}\|u(t)\|_{p_0}^{p_0}\right)^{\frac{1}{p_0}}\\
    &\leq C_{\|u_0\|_\infty},
\end{align*}
where Lemma \ref{lem gamma} was used in the last step.
\hfill $\Box$\end{proof}

\begin{thm}
 Under the assumption of Theorem \ref{Thm main}, the system (\ref{main}) admits a unique global mild solution.
\end{thm}

\begin{proof}
Let $(u,v,\tau)$ be the maximal local solution of the system (\ref{main}). Applying Markov's inequality along with Lemma \ref{lem 4.10}, we find that
\begin{align*}
    \mathbb{P}(\tau\leq T) =& \lim_{m\rightarrow\infty}\mathbb{P}(\tau_m\leq T) \nonumber\\
    \leq & \lim_{m\rightarrow\infty}\mathbb{P}\left( \sup_{t\in[0,T\wedge\tau]}\|u(t)\|_\infty \geq m \right) \nonumber\\
    \leq & \lim_{m\rightarrow\infty}\frac{1}{m}\mathbb{E}\sup_{t\in  [0,T\wedge\tau]}\|u(t)\|_\infty \nonumber\\
    =& 0.
\end{align*}
Thus, we conclude that $\tau > T$ for all $T > 0$, establishing the global existence of solutions. The uniqueness was proved in Theorem {\ref{uniqueness}}.
\hfill $\Box$\end{proof}

\section{The case of nonlinear noise}\label{section5}
\setcounter{equation}{0}
In this section, we will show that there exists a unique solution to the stochastic chemotaxis system (\ref{main2}), that is, complete the proof of Theorem 2.3. We first recall the definition of local mild solutions.
\begin{defn}\label{def local solution'}
	We say that $(u,v,\tau)$ is a local mild solution to the system (\ref{main2}) if
	\begin{enumerate}
    \item$\tau$ is a stopping time, and $(u, v)$ is a progressively measurable stochastic process that takes values in $C(\bar{\mathcal{O}})\times W^{1,\infty}(\mathcal{O})$.

    \item There exists a non-decreasing sequence of stopping times $\{\tau_l, l \geq 1\}$ such that $\tau_l \uparrow \tau$ almost surely as $l \to \infty$, and for each $l$, the pair $(u(t), v(t))$ satisfies:
   \begin{align*}
    u(t) &= e^{-tA} u_0 - \int_0^{t} e^{-(t - s) A} \left( \nabla \cdot \left( \chi u(s) \nabla v(s) \right)-g(u(s))\right)\, ds\\
    &\quad + \sum_{i=1}^{\infty}b_i\int_0^t e^{-(t - s) A}\|u(s)\|_q^ru(s)  \, dW^i_s,\\
    v(t) &= G * u(t),
    \end{align*}
    for all $t \leq \tau_l$.
\end{enumerate}
\end{defn}
Similarly as in Section 3, let $u_m$ be the solution of the localized equation:
\begin{align*}
du &= \Delta u  \, dt - \theta_m(\|u\|_{\Upsilon^u_t})\nabla \cdot (\chi u \nabla (G * u))  \, dt + \theta_m(\|u\|_{\Upsilon^u_t})g(u) \, dt \nonumber\\
&\quad \quad + \sum_{i=1}^{\infty}b_i\theta_m(\|u\|_{\Upsilon^u_t})\|u(t)\|_q^ru(t)\,  \, dW^i_t. \quad
\end{align*}

Recall
\begin{align*}
    \tau_m = \inf \left\{ t > 0 \mid \|u_m\|_{\Upsilon^u_t} \geq m \right\}.
\end{align*}
Observe that $\|u(s)\|_q^r$ is bounded above by $Cm^r$ before the stopping time $\tau_m$. The proof of the following theorem is entirely similar to that of Theorem 3.1, Theorem 3.2. in Section 3. It is omitted.

\begin{thm}\label{thm l'}
Suppose that $\sum_{i=1}^\infty b_i^2< \infty$, for any $g\in C^1([0,\infty))$  and nonnegative $u_0\in C(\bar{\mathcal{O})}$, there exists a unique  nonnegative local maximal mild solution  $(u,v,\tau)$  to the system (\ref{main2}).
\end{thm}

To prove Theorem \ref{Thm main2},  it suffices to demonstrate the global existence of the solution. Similar to the discussion at the beginning of Section \ref{section4}, we have
$$
\limsup_{t \nearrow \tau} \sup_{s \in [0, t]} \left(\|u(s)\|_\infty + \|v(s)\|_{1, \infty}\right) = \infty \text{ on } \{\tau < \infty\} \text{ a.s.}.
$$

The following It\^{o} formula for $\|u(t)\|_{p}^{p}$ can be proved in a similar manner as Theorem \ref{thm 4.8} using approximations.
\begin{thm}\label{thm 5.1}
    Let $(u, v, \tau_m)$ be a local mild solution of system (\ref{main2}), with $\tau_m$ defined by
    \begin{align*}
        \tau_m = \inf \{ t > 0 \mid \|u\|_{\Upsilon^u_t} \geq m \}.
    \end{align*}
    Then, for $p\geq 2$, the following It\^{o} formula for $\|u(t)\|_{p}^{p}$ holds $\mathbb{P}$-a.s. for all $t \in [0, \tau_m]$,
    \begin{align}\label{ito’}
        &\|u(t)\|_{p}^{p} - \|u_0\|_{p}^{p} +p(p - 1)  \int_0^{t}\|\nabla u^{\frac{p}{2}}(s)\|_2^2 \, ds \nonumber \\
        &= p(p - 1) \chi \int_0^{t} \int_\mathcal{O} u^{p - 1}(s) \nabla u(s) \cdot \nabla G * u(s) \, dx \, ds \nonumber \\
        &\quad + p \int_0^{t} \int_\mathcal{O} u^{p - 1}(s) g(u(s)) \, dx \, ds \nonumber \\
        &\quad + p \sum_{i=1}^\infty \int_0^{t} \int_\mathcal{O} b_i\|u(s)\|_q^{r}|u(s)|^{p} \, dx \, dW^i_s \nonumber \\
        &\quad + \frac{p(p - 1)}{2} \int_0^{t} \|u(s)\|_q^{2r}\int_\mathcal{O} \sum_{i=1}^\infty b_i^2 |u(s)|^{p} \, dx \, ds.
    \end{align}
\end{thm}

We are now ready to prove the global existence of equation (\ref{main2}). First we have the following estimates whose proof is inspired by \cite{HLL}.
\begin{lem}\label{lem gamma'}
Under the assumption of Theorem \ref{Thm main2}, For $T>0$, there exists a constant $C_{q,\|u_0\|_\infty} > 0$ such that for $0<\eta<\frac{1}{q}$, the following inequality holds:
\begin{equation*}
\mathbb{P}\left( \sup_{t\in [0,T\wedge\tau]}\|u(t)\|_{q}^{q}\geq R\right)\leq \frac{C_{q,\|u_0\|_\infty}}{R^\eta}.
\end{equation*}
\end{lem}
\begin{proof}
     Let $t\in [0, T]$. Using (\ref{ito’}) with $p=q$, applying the It\^{o} formula to the Lyapunov function $V(r)=\log(1+r)$, we obtain that, for $t\in [0,\tau_m]$
     \begin{align}\label{lem 1'}
         &\log(1+\|u(t)\|_{q}^{q})- \log(1+\|u_0\|_{q}^{q})+q(q - 1)  \int_0^{t}\frac{\|\nabla u^{\frac{q}{2}}(s)\|_2^2 }{1+\|u(s)\|_{q}^{q}}\, ds\nonumber\\
         &= q(q - 1) \chi \int_0^{t} \frac{\int_\mathcal{O} u^{q - 1}(s) \nabla u(s) \cdot \nabla G * u(s) \, dx}{1+\|u(s)\|_{q}^{q}} \, ds \nonumber \\
        &\quad + q \int_0^{t} \frac{\int_\mathcal{O} u^{q - 1}(s) g(u(s)) \, dx }{1+\|u(s)\|_{q}^{q}} \, ds \nonumber \\
        &\quad + q \sum_{i=1}^\infty \int_0^{t} \frac{\|u(s)\|_q^{r}\int_\mathcal{O}  b_i |u(s)|^{q} \, dx}{1+\|u(s)\|_{q}^{q}} \, dW^i_s \nonumber \\
        &\quad + \frac{q(q - 1)}{2} \int_0^{t} \frac{\|u(s)\|_q^{2r}\int_\mathcal{O} \sum_{i=1}^\infty b_i^2 |u(s)|^{q} \, dx}{1+\|u(s)\|_{q}^{q}} \, ds\nonumber\\
        &\quad - \frac{q^2}{2} \sum_{i=1}^\infty\int_0^{t} \frac{\|u(s)\|_q^{2r}\left(\int_\mathcal{O}  b_i |u(s)|^{q} \, dx\right)^2}{\left(1+\|u(s)\|_{q}^{q}\right)^2} \, ds.
     \end{align}
Multiplying the second equation in (\ref{main2}) by $u^{q}$, we obtain
\begin{align}\label{lem 2'}
{q}\int_\mathcal{O} u^{{q}-1}(s)\nabla u\cdot\nabla G*u(s) \, dx = -\int_\mathcal{O} u^{q}(s) G*u(s) \, dx + \int_\mathcal{O} u^{{q}+1}(s) \, dx.
\end{align}

Substituting (\ref{lem 2'}) into (\ref{lem 1'}), keeping in mind that $u(t)$ is non-negative and moving terms around, we obtain
\begin{align}\label{lem 3'}
         &\log(1+\|u(t)\|_{q}^{q})+q(q - 1)  \int_0^{t}\frac{\|\nabla u^{\frac{q}{2}}(s)\|_2^2 }{1+\|u(s)\|_{q}^{q}}\, ds+(q - 1) \chi \int_0^{t} \frac{\int_\mathcal{O} u^{q}(s) G * u(s) \, dx}{1+\|u(s)\|_{q}^{q}} \, ds\nonumber\\
         &\leq \log(1+\|u_0\|_{q}^{q})+  \int_0^{t} \frac{(q - 1) \chi \int_\mathcal{O} u^{q + 1}(s) \, dx+q \int_\mathcal{O} u^{q - 1}(s) g(u(s)) \, dx}{1+\|u(s)\|_{q}^{q}} \, ds \nonumber \\
        &\quad - \frac{q}{2} \sum_{k=1}^\infty b_i^2\int_0^{t} \frac{\|u(s)\|_q^{2q+2r}}{\left(1+\|u(s)\|_{q}^{q}\right)^2} \, ds. \nonumber \\
        &\quad + \frac{q(q - 1)}{2} \sum_{k=1}^\infty b_i^2\int_0^{t} \frac{\|u(s)\|_q^{q+2r}}{\left(1+\|u(s)\|_{q}^{q}\right)^2} \, ds\nonumber\\
        &\quad + q \sum_{i=1}^\infty \int_0^{t} \frac{\|u(s)\|_q^{r}\int_\mathcal{O}  b_i |u(s)|^{q} \, dx}{1+\|u(s)\|_{q}^{q}} \, dW^i_s\nonumber\\
        &=:\log(1+\|u_0\|_{q}^{q})+I_1+I_2+I_3+M(t).
     \end{align}
Before we proceed, recall the Gagliardo-Nirenberg interpolation inequality in the  bounded smooth domain $\mathcal{O}\subset \mathbb{R}^2$:
Assume that $1\leq q,r\leq\infty$ ,
 $k, j \in \mathbb{N}$ with $j <k$ and $p\in \mathbb{R}$ such that
\begin{align*}
    \frac{1}{p}=\frac{j}{2}+\theta (\frac{1}{r}-\frac{k}{2})+(1-\theta)\frac{1}{q},\quad\frac{j}{k}\leq \theta <1.
\end{align*}
 Then there exists a constant $C$ depending only on $k$, $q$, $ r$, $\theta$ and $\mathcal{O}$ such that:
 \begin{align}\label{5.4}
     \|u\|_{W^{j,p}(\mathcal{O})}\leq C\|u\|_{W^{k,r}(\mathcal{O})}^\theta\|u\|_q^{1-\theta}+ C\|u\|_q,\quad u\in L^q(\mathcal{O})\cap W^{k,r}(\mathcal{O}).
 \end{align}

  Using the condition (\ref{g'}) and the fact that $\|u\|_p^p\leq C(\|u\|_q^q+1)$ for all $p<q$, we have
\begin{align}
    I_1&\leq  \int_0^{t} \frac{(q - 1) \chi \int_\mathcal{O} u^{q + 1}(s) \, dx+q c_2 \int_\mathcal{O} u^{q - 1}(s)  \, dx+q \mu ' \int_\mathcal{O} u^{q+n - 1}(s)  \, dx}{1+\|u(s)\|_{q}^{q}} \, ds\nonumber\\
    &\leq \int_0^{t} \frac{C  \int_\mathcal{O} u^{q - 1 + 2\vee n}(s) \, dx+C}{1+\|u(s)\|_{q}^{q}} \, ds\nonumber\\
    &\leq C+C \int_0^{t} \frac{ \|u^\frac{q}{2}(s)\|_{W^{1,2}(\mathcal{O})}^\frac{2(2\vee n-1)}{q}\|u^\frac{q}{2}(s)\|_2^2+\|u^\frac{q}{2}(s)\|_2^{\frac{2(q+2\vee n-1)}{q}}}{1+\|u(s)\|_{q}^{q}} \, ds\nonumber\\
    &\leq C+\int_0^{t} \frac{\epsilon_1  \|u^\frac{q}{2}(s)\|_{W^{1,2}(\mathcal{O})}^2+ C_{\epsilon_1} \|u(s)\|_q^\frac{q^2}{q-2\vee n+1}+\|u(s)\|_q^{q+2\vee n-1}}{1+\|u(s)\|_{q}^{q}} \, ds\nonumber\\
    &\leq C+\int_0^{t} \frac{\epsilon_1  \|u^\frac{q}{2}(s)\|_{W^{1,2}(\mathcal{O})}^2+ \epsilon_2\|u(s)\|_q^{q+2r}+C_{\epsilon_1,\epsilon_2}}{1+\|u(s)\|_{q}^{q}} \, ds.\nonumber
\end{align}
Here, in the third inequality, we utilized the Gagliardo-Nirenberg interpolation inequality by setting the parameters $(p, q, r, j, k, \theta)$ to $\left(\frac{2(q + 2 \vee n - 1)}{q}, 2, 2, 0, 1, \frac{2 \vee n - 1}{q + 2 \vee n - 1}\right)$ in (\ref{5.4}), and in the fourth and fifth inequalities, we applied Young's inequality along with the assumption \textbf{(A-1)}.

Due to the fact that $2r\leq q$, we have
\begin{align*}
    I_3\leq C\int_0^{t} \frac{\|u(s)\|_q^{2q}}{\left(1+\|u(s)\|_{q}^{q}\right)^2} \, ds\leq C.
\end{align*}

Substituting the above two inequalities into (\ref{lem 3'}), we get
\begin{align*}
         &\log(1+\|u(t)\|_{q}^{q})+(q^2 - q -\epsilon_1)  \int_0^{t}\frac{\|\nabla u^{\frac{q}{2}}(s)\|_2^2 }{1+\|u(s)\|_{q}^{q}}\, ds+(q - 1) \chi \int_0^{t} \frac{\int_\mathcal{O} u^{q - 1}(s) G * u(s) \, dx}{1+\|u(s)\|_{q}^{q}} \, ds\nonumber\\
         &\leq C+ \log(1+\|u_0\|_{q}^{q})+  \int_0^{t} \frac{\epsilon_2\|u(s)\|_q^{q+2r}+C_{\epsilon_1,\epsilon_2}(1+\|u(s)\|_{q}^{q})}{\left(1+\|u(s)\|_{q}^{q}\right)^2} \, ds \nonumber \\
        &\quad - \left(\frac{q}{2} \sum_{k=1}^\infty b_i^2-\epsilon_2\right)\int_0^{t} \frac{\|u(s)\|_q^{2q+2r}}{\left(1+\|u(s)\|_{q}^{q}\right)^2} \, ds +M(t)\nonumber\\
        &\leq C+ \log(1+\|u_0\|_{q}^{q})+ M_t - \frac{\left(\frac{q}{2} \sum_{k=1}^\infty b_i^2-\epsilon_2\right)}{q^2\sum_{k=1}^\infty b_i^2}\langle M\rangle_t,
     \end{align*}
     where $\langle M\rangle_t$ represents the quadratic variation of $M_t$, the constant $C$ actually depends on $T, \epsilon_1, \epsilon_2$ etc.  Choosing $\epsilon_1=\frac{q^2-q}{2}$ and $\epsilon_2=\frac{q-q^2\eta}{2}\sum_{k=1}^\infty b_i^2$, we deduce that
    \begin{align*}
         \log(1+\|u(t)\|_{q}^{q})
        \leq C+ \log(1+\|u_0\|_{q}^{q})+ M_t - \frac{\eta}{2}\langle M\rangle_t, \quad t\leq T.
     \end{align*}
This  implies that for $R>0$,
\begin{align*}
   & \mathbb{P}\left( \sup_{t\in [0,T\wedge\tau_m]}\|u(t)\|_{q}^{q}\geq R\right)\nonumber\\
    &\leq\mathbb{P}\left( \eta\log\sup_{t\in [0,T\wedge\tau_m]}\|u(t)\|_{q}^{q}\geq \eta\log R\right)\nonumber\\
    &\leq \mathbb{P}\left( \sup_{t\in [0,T\wedge\tau_m]}\left(\eta M_t - \frac{\eta^2}{2}\langle M\rangle_t\right)\geq \eta (\log R-C- \log(1+\|u_0\|_{q}^{q}))\right)\nonumber\\
    &\leq \mathbb{P}\left( \sup_{t\in [0,T\wedge\tau_m]}\exp\left\{\eta M_t - \frac{\eta^2}{2}\langle M\rangle_t\right\}\geq \exp\left\{\eta(\log R-C- \log(1+\|u_0\|_{q}^{q}))\right\}\right).
\end{align*}
Then, the maximal supermartingale inequality
yields that
\begin{align*}
   & \mathbb{P}\left( \sup_{t\in [0,T\wedge\tau_m]}\|u(t)\|_{q}^{q}\geq R\right)\leq \exp\left\{-\eta(\log R-C- \log(1+\|u_0\|_{q}^{q}))\right\}\leq \frac{C}{R^\eta}.
\end{align*}
Let $m\rightarrow \infty $ to complete the proof.
\hfill $\Box$\end{proof}

\begin{lem}\label{lem 4.10'}
Under the assumption of Theorem \ref{Thm main2}, for $T>0$, there exists constants $C_{\|u_0\|_\infty,T}>0$ and $0<\gamma<\frac{1}{2(2\vee n \vee (r+1))}$ such that
\begin{align*}
    \mathbb{E}\sup_{t\in  [0,T\wedge\tau]}\|u(t)\|_\infty^\gamma \leq C_{\|u_0\|_\infty,T}.
\end{align*}
\end{lem}
\begin{proof}
     Similar to the proof of  equation (\ref{Phi}), we can derive that for $t<\tau$,
\begin{align*}
    \|u(t)\|_\infty \leq & \|e^{-tA}u_0\|_\infty + C \left( \int_0^{t}\|(A_{q}+1)^\beta e^{-(t-s)A}\nabla\cdot(u\nabla G*u)(s)\|_{q} \, ds \right) \nonumber\\
    & + \left(\int_0^{t}\|(A_{\frac{q}{n}}+1)^{\beta'} e^{-(t-s)A}g(u(s))\|_\frac{q}{n} \, ds \right)\nonumber\\
    &+ \|\sum_{i=1}^{\infty}b_i\int_0^t e^{-(t - s) A}\|u(s)\|_q^ru(s)  \, dW^i_s\|_{\infty}.
\end{align*}
Here, due to the assumption \textbf{(A-2)}, $q>2(2\vee n-1)>n$ and  $\beta'\in (\frac{n}{q},1)$ is chosen so that
$D((A_\frac{q}{n}+1)^{\beta'})\hookrightarrow C(\Bar{\mathcal{O}})$. Consequently,
\begin{align}\label{4 I'}
    \|u(t)\|_\infty^{\gamma} \leq & \|e^{-tA}u_0\|_\infty^{\gamma} + C\left( \int_0^{t}\|(A_{q}+1)^\beta e^{-(t-s)A}\nabla\cdot(u\nabla G*u)(s)\|_{q} \, ds \right)^{\gamma}\nonumber\\
    &+\left(\int_0^{t}\|(A_{\frac{q}{n}}+1)^{\beta'} e^{-(t-s)A}g(u(s))\|_\frac{q}{n} \, ds \right)^\gamma \nonumber\\
    & + C\|\sum_{i=1}^{\infty}b_i\int_0^t e^{-(t - s) A}\|u(s)\|_q^ru(s)  \, dW^i_s\|_{\infty}^{\gamma} \nonumber\\
    \leq & C\|u_0\|_\infty^{\gamma} + I_1(t) + I_2(t)+I_3(t).
\end{align}

By (\ref{A2}) and the assumption ({\bf (A-2)}), for $\epsilon\in (0,\frac{1}{2}-\beta)$ and $\epsilon'\in (0,1-\beta')$, we estimate $I_1(t),I_2(t)$ as follows:
\begin{align*}
    I_1(t) \leq & C_{\epsilon}\left(\int_0^{t} (t-s)^{-\beta-\frac{1}{2}-\epsilon} \|(u\nabla G*u)(s)\|_{q} \, ds \right)^{\gamma} \nonumber\\
    \leq & C_{\epsilon}\left(\int_0^{t} (t-s)^{-\beta-\frac{1}{2}-\epsilon} \|u(s)\|_{q}^2 \, ds \right)^{\gamma} \nonumber\\
    \leq & C \sup_{s\in[0,{t}]} \|u(s)\|_{q}^{2\gamma} t^{\gamma(\frac{1}{2}-\beta-\epsilon)},
\end{align*}
\begin{align*}
    I_2(t) \leq & C_{\epsilon'}\left(\int_0^{t}\|(A_{\frac{q}{n}}+1)^{\beta'} e^{-(t-s)A}g(u(s))\|_\frac{q}{n} \, ds \right)^\gamma \nonumber\\
    \leq & C_{\epsilon'}\left(\int_0^{t}  \|1+u(s)\|_{q}^n \, ds \right)^{\gamma} \nonumber\\
    \leq & C \sup_{s\in[0,{t}]} \|u(s)\|_{q}^{n\gamma}.
\end{align*}

To estimate $I_3$, we apply Lemma \ref{M}, choosing $(\eta,r,q)$ in (\ref{Mt}) as $(\gamma,\infty,\frac{q}{2})$, to get
\begin{align*}
    \mathbb{E}\sup_{t\in[0,T\wedge \tau]} I_3(t) \leq C \sum_{i=1}^\infty b_i^2 \mathbb{E}\sup_{t\in[0,T\wedge\tau]} \|u(s)\|_{q}^{r\gamma+\gamma}.
\end{align*}

Now, applying the estimates we obtained for $I_1(t)$, $I_2(t)$ and $I_3(t)$, taking the expectation in (\ref{4 I'}) we deduce that
\begin{align}\label{gamma}
    \mathbb{E}\sup_{t\in  [0,T\wedge\tau]}\|u(t)\|_\infty^\gamma &\leq C+C\mathbb{E}\sup_{t\in[0,T\wedge\tau]}\|u(s)\|_q^{2\gamma}+C\mathbb{E}\sup_{t\in[0,T\wedge\tau]}\|u(s)\|_q^{n\gamma}+C\mathbb{E}\sup_{t\in[0,T\wedge\tau]}\|u(s)\|_q^{(r+1)\gamma}\nonumber\\
    &\leq C+C\mathbb{E}\sup_{t\in[0,T\wedge\tau]}\|u(s)\|_q^{\gamma(2\vee n \vee (r+1))}.
\end{align}
Since $\gamma(2\vee n \vee (r+1))<\frac{1}{2} $, we can find $\eta_0\in (0,\frac{1}{q})$ such that
\begin{align*}
    \beta:=\frac{\eta_0q}{\gamma(2\vee n \vee (r+1))}>2.
\end{align*}
By  Lemma \ref{lem gamma'}, we have
\begin{equation*}
\mathbb{P}\left( \sup_{t\in [0,T\wedge\tau]}\|u(t)\|_{q}^{q}\geq R\right)\leq \frac{C}{R^{\eta_0}}.
\end{equation*}
Thus, we get
\begin{equation*}
\mathbb{P}\left( \sup_{t\in [0,T\wedge\tau]}\|u(t)\|_{q}^{\gamma(2\vee n \vee (r+1))}\geq R\right)\leq \frac{C}{R^\beta}.
\end{equation*}
This implies that
\begin{align}
\mathbb{E}\sup_{t\in[0,T\wedge\tau]}\|u(s)\|_q^{\gamma(2\vee n \vee (r+1))}&\leq \sum_{i=0}^\infty (i+1)\mathbb{P}\left( i+1\geq\sup_{t\in [0,T\wedge\tau]}\|u(t)\|_{q}^{\gamma(2\vee n \vee (r+1))}\geq i\right)\nonumber\\
&\leq \sum_{i=0}^\infty (i+1)\frac{C}{i^\beta}\leq C.\nonumber
\end{align}

The proof is complete in view of  (\ref{gamma}).
\hfill $\Box$\end{proof}
\begin{thm}
 Under the assumption of Theorem \ref{Thm main2}, the system (\ref{main2}) admits a unique global mild solution.
\end{thm}

\begin{proof}
Let $(u,v,\tau)$ be the maximal local solution of the system (\ref{main2}). Applying Markov's inequality along with Lemma \ref{lem 4.10'}, we find that
\begin{align*}
    \mathbb{P}(\tau\leq T) =& \lim_{m\rightarrow\infty}\mathbb{P}(\tau_m\leq T) \nonumber\\
    \leq & \lim_{m\rightarrow\infty}\mathbb{P}\left( \sup_{t\in[0,T\wedge\tau]}\|u(t)\|_\infty \geq m \right) \nonumber\\
    =&\lim_{m\rightarrow\infty}\mathbb{P}\left( \sup_{t\in[0,T\wedge\tau]}\|u(t)\|_\infty^{\gamma } \geq m^{\gamma } \right) \nonumber\\
    \leq & \lim_{m\rightarrow\infty}\frac{1}{m^{\gamma }}\mathbb{E}\sup_{t\in [0,T\wedge\tau]}\|u(t)\|_\infty^{\gamma } \nonumber\\
    =& 0.
\end{align*}
Since $T$ is arbitrary, we have $P(\tau=\infty)=1$, establishing the global existence of solutions. The uniqueness was given in Theorem {\ref{thm l'}}.
\hfill $\Box$\end{proof}

\vskip 0.5cm

\noindent {\bf Acknowledgments}.  This work is partly supported by National Key R\&D
Program of China (No.2022YFA1006001) and by National Natural Science Foundation of China(No. 12131019, No. 12001516, No. 11721101, No. 12371151),  the Fundamental Research Funds for the Central Universities (No. WK3470000031, No. WK0010000081).

\bibliographystyle{cas-model2-names}

\end{document}